\documentclass[final,2p,times]{elsarticle}

\usepackage{bbm}
\usepackage{amsxtra,amsmath,amssymb,amsfonts,exscale,amsthm,amscd}
\usepackage{xcolor}
\usepackage{float}
\usepackage{epic}
\usepackage{latexsym}
\usepackage{geometry}
\usepackage{graphicx}
\usepackage{subfigure}
\usepackage{natbib}
\usepackage{txfonts}
\usepackage{psfrag}
\usepackage{showkeys}
\usepackage{pifont}
\usepackage{ifthen}
\usepackage{ifpdf}
\usepackage{slashbox}
\usepackage{tabularx}
\usepackage{booktabs}
\usepackage{dcolumn}
\usepackage{multirow}
\usepackage{wrapfig}
\usepackage{multicol}
\usepackage{verbatim}
\usepackage[shortlabels]{enumitem}
\usepackage{layouts}
\usepackage{stmaryrd}
\usepackage{xcolor}
\usepackage{tikz}
\usepackage{diagbox}
\usepackage{mathrsfs}
\usepackage{ulem}
\usepackage{bm}

\providecommand*{\Dist}[2]{\operatorname{dist}({#1};{#2})}   
\providecommand*{\Dist}[2]{\Dist{#1}{#2}}
\providecommand{\Det}{\operatorname{det}}                    

\providecommand{\argmin}{\operatorname*{argmin}}  

\newcommand{\Bc}{{\boldsymbol{c}}}
\newcommand{\Bd}{{\boldsymbol{d}}}
\newcommand{\Be}{{\boldsymbol{e}}}
\newcommand{\Bf}{{\boldsymbol{f}}}
\newcommand{\Bg}{{\boldsymbol{g}}}

\newcommand{\Bn}{{\boldsymbol{n}}}

\newcommand{\Bp}{{\boldsymbol{p}}}

\newcommand{\Bu}{{\boldsymbol{u}}}
\newcommand{\Bv}{{\boldsymbol{v}}}
\newcommand{\Bw}{{\boldsymbol{w}}}
\newcommand{\Bx}{{\boldsymbol{x}}}
\newcommand{\By}{{\boldsymbol{y}}}

\newcommand{\BA}{{\boldsymbol{A}}}

\newcommand{\BC}{{\boldsymbol{C}}}

\newcommand{\BF}{{\boldsymbol{F}}}
\newcommand{\BG}{{\boldsymbol{G}}}
\newcommand{\BH}{{\boldsymbol{H}}}
\newcommand{\BI}{{\boldsymbol{I}}}

\newcommand{\BL}{{\boldsymbol{L}}}

\newcommand{\BP}{{\boldsymbol{P}}}

\newcommand{\BU}{{\boldsymbol{U}}}
\newcommand{\BV}{{\boldsymbol{V}}}
\newcommand{\BW}{{\boldsymbol{W}}}
\newcommand{\BX}{{\boldsymbol{X}}}

\newcommand{\etabf}{\boldsymbol{\eta}}

\newcommand{\xibf}{\boldsymbol{\xi}}

\newcommand{\varphibf}{\boldsymbol{\varphi}}
\newcommand{\chibf}{\boldsymbol{\chi}}

\newcommand{\Lambdabf}{\boldsymbol{\Lambda}}

\newcommand{\Psibf}{\boldsymbol{\Psi}}

\newcommand{\Phibf}{\boldsymbol{\Phi}}

\newcommand{\Ca}{\mathcal{A}}

\newcommand{\Ce}{\mathcal{E}}
\newcommand{\Cf}{\mathcal{F}}

\newcommand{\Ci}{\mathcal{I}}

\newcommand{\Cp}{\mathcal{P}}

\newcommand{\Ct}{\mathcal{T}}

\newcommand{\Cw}{\mathcal{W}}

\newcommand{\bbI}{\mathbb{I}}
\newcommand{\bbJ}{\mathbb{J}}

\newcommand{\bbL}{\mathbb{L}}

\newcommand{\bbR}{\mathbb{R}}

\newcommand*{\N}[1]{\left\|{#1}\right\|}     
\newcommand*{\TN}[1]{\left\|\left|{#1}\right|\right\|}     
\newcommand*{\SN}[1]{\left|{#1}\right|}      

\newcommand*{\Lp}[2][\defaultdomain]{L^{#2}({#1})}
\newcommand*{\Lpv}[2][\defaultdomain]{\BL^{#2}({#1})}
\newcommand*{\NLp}[3][\defaultdomain]{\N{#2}_{\Lp[#1]{#3}}}
\newcommand*{\NLpv}[3][\defaultdomain]{\N{#2}_{\Lpv[#1]{#3}}}

\newcommand*{\Ltwo}[1][\defaultdomain]{\Lp[#1]{2}}
\newcommand*{\Ltwov}[1][\defaultdomain]{\Lpv[#1]{2}}
\newcommand*{\NLtwo}[2][\defaultdomain]{\NLp[#1]{#2}{2}}
\newcommand*{\NLtwov}[2][\defaultdomain]{\NLpv[#1]{#2}{2}}

\newcommand*{\Linf}[1][\defaultdomain]{L^{\infty}({#1})}

\newcommand*{\NLinf}[2][\defaultdomain]{\N{#2}_{\Linf[{#1}]}}

\newcommand*{\Hm}[2][\defaultdomain]{H^{#2}({#1})}
\newcommand*{\Hmv}[2][\defaultdomain]{\BH^{#2}({#1})}

\newcommand*{\Hone}[1][\defaultdomain]{\Hm[#1]{1}}
\newcommand*{\Honev}[1][\defaultdomain]{\Hmv[#1]{1}}

\newcommand*{\NHone}[2][\defaultdomain]{{\N{#2}}_{\Hone[{#1}]}}

\newcommand*{\SNHone}[2][\defaultdomain]{{\SN{#2}}_{\Hone[{#1}]}}
\newcommand*{\SNHonev}[2][\defaultdomain]{{\SN{#2}}_{\Honev[{#1}]}}

\newcommand*{\jump}[2][]{\left \llbracket{#2}\right\rrbracket_{#1}}

\newcommand*{\avg}[2][]{\left\{\hskip -3.5pt\left\{{#2}
	\right\}\hskip -3.5pt\right\}_{#1}}

\newcommand{\D}{\mathrm{d}}

\newcommand{\ol}{\overline}
\newcommand{\ul}{\underline}

\newcommand{\be}{\begin{eqnarray}}
	\newcommand{\ee}{\end{eqnarray}}

\newcommand{\ben}{\begin{eqnarray*}}
	\newcommand{\een}{\end{eqnarray*}}

\def\cdotbf{\bm{\cdot}}
\newtheorem{lemma}[subsection]{\sc Lemma}
\newtheorem{theorem}[subsection]{\sc Theorem}

\newtheorem{remark}[subsection]{\sc Remark}

\newtheorem{corollary}[subsection]{\sc Corollary}

\begin{document}

\begin{frontmatter}
	
	\title{A new framework of high-order unfitted finite element methods using ALE maps for moving-domain problems}
	
	\author[add1]{Wenhao Lu}
	\ead{luwenhao@lsec.cc.ac.cn}
	\author[add3]{Chuwen Ma\corref{cor}\fnref{ma}}
	\ead{chuwenii@lsec.cc.ac.cn}
	\author[add1,add2]{Weiying Zheng\fnref{zwy}}
	\ead{zwy@lsec.cc.ac.cn}
	\address[add1]{LSEC, NCMIS, Institute of Computational Mathematics and Scientific/Engineering Computing, Academy of Mathematics and Systems Science, Chinese Academy of Sciences, Beijing, 100190, China.}
	\address[add2]{School of Mathematical Science, University
		of Chinese Academy of Sciences, Beijing, 100049, China.}
	\address[add3]{Institute of Natural Sciences, Shanghai Jiao Tong University, Shanghai 200240, China.}
	\cortext[cor]{Corresponding author}
	
	\fntext[ma]{This author was supported in part
		by China Postdoctoral Science Foundation (No. 2023M732248) and Postdoctoral Innovative Talents Support Program (No. 20230219).}
	
	\fntext[zwy]{This author was supported in part by National Key R \& D
		Program of China 2019YFA0709600 and 2019YFA0709602
		and by the National Science Fund for Distinguished Young Scholars 11725106.}

	\begin{abstract}
 As a sequel to our previous work [C. Ma, Q. Zhang and W. Zheng, SIAM J. Numer. Anal., 60 (2022)], [C. Ma and W. Zheng, J. Comput. Phys. 469 (2022)], this paper presents a generic framework of arbitrary Lagrangian-Eulerian unfitted finite element (ALE-UFE) methods for partial differential equations (PDEs) on time-varying domains. The ALE-UFE method has a great potential in developing high-order unfitted finite element methods. The usefulness of the method is demonstrated by a variety of moving-domain problems, including a linear problem with explicit velocity of the boundary (or interface),
a PDE-domain coupled problem, and a problem whose domain has a topological change. Numerical experiments show that optimal convergence is achieved by both third- and fourth-order methods on domains with smooth boundaries, but is deteriorated to the second order when the domain has topological changes.
	\end{abstract}
	
	\begin{keyword}
		Arbitrary Lagrangian-Eulerian unfitted finite element (ALE-UFE) method,  moving domain problems,  high-order schemes.
		\vspace{2mm}
		
	\end{keyword}
\end{frontmatter}
	
 \section{Introduction}
Multiphase flows with time-varying domains or free-surface fluids have an extremely wide range of applications in science and engineering.  
Transient deformations of material regions pose a great challenge to the design of high-order numerical methods for solving such problems. In terms of the relative position of a moving domain to its partition mesh, current numerical methods can be roughly classified into two regimes. 
In body-fitted methods, the mesh is arranged to follow the moving phase, and the implementation of boundary conditions becomes easy.
In the popular arbitrary Lagrangian-Eulerian (ALE) approach \cite{hir74,don82,for99,far01,geu03,hug16},
the mesh velocity can be chosen independent of local fluid velocity.
One needs to transform the problem from a moving domain to a fixed reference domain through an arbitrary mapping and further mesh the reference domain.
The method can also be used in combination with space-time Galerkin formulations\cite{mas97,tez92}.
However, these conveniences incur the cost of mesh regeneration and data migration across the entire computational domain at each time step. Another major concern of body-fitted methods is how to maintain high accuracy in the presence of abrupt motions or large deformations of the phase \cite{rich17}.

In the other regime of unfitted methods, the mesh for the bulk phase is fixed while the moving interface is allowed to cross the fixed mesh, resulting in cut cells near the interface where the degrees of freedom are doubled with additional penalty terms or basis functions are modified to enforce the interface conditions weakly. 
Unfitted methods may encounter significant challenges for dynamic interface problems.
Since the computational domain is time-varying, 
traditional methods for time integration may not be applicable \cite{fri09,zun13}.
To overcome this difficulty, the space-time method in \cite{leh13}
combines a discontinuous Galerkin technique in time with an extended finite element method. The immersed finite element method exploits the invariant degrees of freedom and uses backward Euler for the time discretization \cite{guo21,adj19}.
Methods in \cite{leh19,lou21,wah20,bur22} extend the discrete solution at each time-step by using a ghost penalty, which enables the use of a backward differentiation formula.
Most recently, von Wahl and Richter developed  an error estimate
for this Eulerian time-stepping scheme for a PDE on a moving domain, where the
domain motion is driven by an ordinary differential equation (ODE) coupled to the
PDE \cite{von23}.
In addition, the characteristic approach is applied in \cite{ma21,ma23,maz21} to 
develop high-order unfitted finite element methods and in \cite{hansbo15} for convection-diffusion problems on time-dependent surfaces.

This paper is inspired by the ALE method, where the mesh velocity is typically derived from boundary motion, which allows the fluid to move with respect to the mesh.
We exploit the ALE map to construct a backward flow map and then approximate the time derivative, which makes it easier to obtain high-order convergence rates even for large deformations. The basic idea behind is to replace the partial time derivative with a material derivative,
\begin{equation*}
	\partial_t u(x,t_n) = \frac{\D}{\D t} u(\Ca^n(x,t),t\big) \big|_{t=t_n}
	-\Bw^n(\Bx,t_n) \cdot \nabla u(x,t_n),
\end{equation*}
where $\Ca^n(x,t)$ is an arbitrary mapping that maps $\Bx \in \Omega_{t_n}$ to 
$y = \Ca^n(x,t)\in \Omega_t$, for any $t\leq t_n$, and $\Bw^n = \partial_t \Ca^n(x,t)$ is the velocity of the moving domain.
The arbitrary feature of the algorithm arises from the fact that this application dose not follow trajectories of bulk fluid particles, but only of boundary fluid particles.
In particular, if $\Bw^n$ coincides with the fluid velocity, 
the method turns out to be the unfitted characteristic finite element (UCFE) method proposed in \cite{ma21,maz21,ma23}. Compared with the UCFE method, the proposed method -- ALE-UFE has two  advantages:
\begin{enumerate}[leftmargin=6mm]
	\item [(i)] 
	In the framework of unfitted  finite element method, it is 
	difficult to calculate the integration of the  numerical solutions from
	early time steps at the present time step, since they do not belong to the present finite element space, especially in nonlinear problems where the moving interface depends on the solution of the equation.  
	The ALE-UFE method overcomes this difficulty by choosing a simple ALE mapping to construct a backward flow map (see section~4.3).
	\item [(ii)] 
	Since the construction of the ALE map requires only the position of the moving interface and not the internal variation of the region as in the UCFE approach, 
	the ALE-UFE method can be applied to more models, 
	including those where the moving region does not maintain the same volume (see section \ref{subsec:Free boundary problems}) and the topology changes
	(see section \ref{sec:Problems with a topological change}).
\end{enumerate}

The main contributions of this work are summarized as follows.

\begin{enumerate}[leftmargin=6mm]
	\item [(a)]
	Based on the heat equation in a time-evolving domain, we propose a new framework of designing high-order unfitted finite element methods with ALE maps. We prove the stability of the numerical solution in the energy norm and establish optimal error estimates, arising from both interface-tracking algorithms and spatial-temporal discretization of the equations.
	\vspace{0.5mm}

	
	
	\item [(b)] The competitive performance of the ALE-UFE method is demonstrated by numerical experiments on largely deforming domains, including a PDE-domain coupled model, a two-fluid model with moving interface, and a problem with topologically changing domain. Optimal convergence of the method is obtained for domains with smooth boundaries or interface, but is deteriorated for domains with topological changes.
\end{enumerate}

The rest of the paper is organized as follows. In section 2, we introduce the forward and backward flow maps. In section 3, we propose the ALE-UFE method for the heat equation on a moving domain, and establish the stability and error estimates of the numerical solution. A general framework of the ALE-UFE method is presented. In section 4, we apply the ALE-UFE method to a PDE-domain coupled problem. In section 5, the ALE-UFE method is applied to a two-phase flow problem. In section 6, we present numerical results to demonstrate optimal convergence of both the third- and fourth-order methods, and apply the method to more challenging problems.

Throughout this paper, vector-valued quantities and matrix-valued quantities are denoted by boldface symbols and blackboard bold symbols, respectively, such as $\Ltwov=\Ltwo^2$ and $\bbL^\infty(\Omega)=\Linf^{2\times 2}$. The notation $f \lesssim g$ means that $f\le Cg$ holds with a constant $C>0$ independent of sensitive quantities, such as the segment size $\eta$ for interface-tracking, the spatial mesh size $h$, the time-step size $\tau$, and the number of time steps $n$. Moreover, 
we use the notations $(\cdot,\cdot)_{\Omega}$ and 
$\left<\cdot ,\cdot\right>_{\Gamma}$
to denote the inner product on $L^2(\Omega)$ and $L^2(\Gamma)$, respectively.

\section{Flow maps}\label{sec:flow maps}

In this section, we introduce the forward boundary map for interface tracking  and the backward flow map for time integration based on a given velocity field which has compact support in space. Throughout the theoretical analysis, we assume the velocity is smooth such that $\Bv\in \BC^r(\bbR^2\times [0,T])$, $r\geq k+1$ with $2\le k\le 4$ being the order of time integration.

\subsection{Forward boundary map}

Suppose the initial domain $\Omega_0\subset\bbR^2$ is bounded and has a $C^r$-smooth boundary $\Gamma_0=\partial\Omega_0$. 
The variation of the domain boundary $\Gamma_t$ has the form 
\begin{equation}\label{eq:Gammat}
	\Gamma_t = \{\BX_F(t;0,\Bx_0):\forall \Bx_0 \in \Gamma_0\},
\end{equation}
where $\BX_F$ is a forward boundary map, defined by
\begin{equation}\label{eq:X}
	\frac{\D }{\D t}\BX_F(t;s,\Bx_s)= \Bv(\BX_F(t;s,\Bx_s),t),\quad
	\BX_F(s;s,\Bx_s)=\Bx_s.
\end{equation}
The moving domain $\Omega_t$ is surrounded by
the boundary $\Gamma_t$, i.e., $\Gamma_t = \partial \Omega_t$.
For theoretical analysis, we assume that  $\Bv$ is $\BC^r$-smooth, thus $\BX_F$ is a diffeomorphism and $\Omega_t$ has same topological properties as $\Omega_0$.  Meanwhile, we also assume that $\Gamma_t$ is $C^r$ smooth for all $t\in [0,T]$.

Consider a uniform partition of the interval $[0,T]$, given by $t_n=n\tau$, $n=0,1,\cdots,N$, where $\tau=T/N$.
Denote the forward boundary maps, the transient boundaries, and the transient domains at time $t_n$, respectively, by 
\begin{equation*}
	\BX_F^{n-1,n} := \BX_F(t_n;t_{n-1},\cdot),\quad 
	\Gamma^n = \Gamma_{t_n},\quad \Omega^n = \Omega_{t_n}, \quad n>0.
\end{equation*}
The well-posedness of \eqref{eq:X} implies that 
$\BX_F^{n-1,n}$: $\Gamma^{n-1}\to \Gamma^n$ 
is one-to-one.

\subsection{Backward flow map based on the ALE map}
\label{sec:backmap}

For fixed $t_n$, we choose $\Omega^{n}$ as a reference domain and are going to define a backward flow map $\Ca^n(\cdot,t)$ and the corresponding fluid velocity $\Bw^n$ by means of ALE map: for $t\leq t_n$,
\begin{equation}\label{def At wn}
	\Ca^n(\Bx,t):\bar \Omega^n\to \bar\Omega_t,\quad 
	\Ca^n(\Bx,t_n) := \Bx,\qquad
	\Bw^n(\Bx,t) = \partial_t \Ca^n(\Bx,t).
\end{equation} 
In practice, we only need the ALE map at discrete time steps. First we construct a discrete boundary map either by 
$\Bg^{n,n-1}(\Bx) =\BX_F(t_{n-1};t_n,\Bx)$ or by the closet point mapping 
$\Bg^{n,n-1}(\Bx) = \argmin\limits_{\By\in \Gamma^{n-1}} |\By-\Bx|$ (see  \cite{hansbo15}), for all $\Bx \in \Gamma^n$.
The backward flow map is defined by the solution to the harmonic equation
\begin{equation}\label{eq:construction of X}
	-\Delta \BX^{n,n-1} = \textbf{0} \quad \text{in}\;\;\Omega^n,\qquad
	\BX^{n,n-1} = \Bg^{n,n-1} \quad \text{on}\;\;\Gamma^{n}.
\end{equation}
The maximum principle implies that $\BX^{n,n-1}$ maps $\ol{\Omega^n}$ to $\ol{\Omega^{n-1}}$.
We then define a multi-step map from $\ol{\Omega^{n}}$ to $\ol{\Omega^{n-i}}$,
$2\leq i\leq k$, by
\begin{equation}\label{eq:multi step}
	\BX^{n,n-i} = \BX^{n-i+1,n-i}\circ \cdots \circ \BX^{n,n-1}.  
\end{equation}

Let $l_n^i\in P_k([t_{n-k},t_n])$ be the basis functions of Lagrange interpolation satisfying $l_n^i(t_{n-j}) = \delta_{ij}$, with $\delta_{ij}$ the Kronecker delta. Define the semi-discrete ALE map by 
\begin{equation}\label{def At k}
	\Ca^{n}_k(\Bx,t) = \sum\limits_{i=0}^k l_n^i(t)\BX^{n,n-i}(\Bx)\quad \forall\, 
	(\Bx,t) \in \ol{\Omega^n} \times [t_{n-k},t_n].
\end{equation}
Clearly $\Ca^n(\cdotbf,t_{n-i})=\Ca^{n}_k(\cdot\,t_{n-i}) = \BX^{n,n-i}$. The artificial fluid velocity is defined by 
\begin{equation}\label{eq:wn}
	\Bw^n_k(\Bx,t) := \partial_t \Ca^n_k(t,\Bx)
	= \sum_{i=0}^k (l_n^i)'(t)\BX^{n,n-i}(\Bx).
\end{equation}

\begin{remark}
	The construction of the ALE map is not unique. For example, one can represent the motion of a domain by considering the domain as elastic or viscoelastic solid, and solve the problem by resorting to the equations of elastic dynamics.  
\end{remark}

\section{The ALE-UFE method}

The purpose of this section is to propose a high-order finite element method for solving PDEs on time-moving domains based on ALE map. For clearness, we first focus on the heat equation and will extend the result to more complex problems in 
sections \ref{sec:nonlinear}--\ref{sec:num}. 

\subsection{The heat equation on a moving domain}

Based on the forward boundary map and backward flow map presented in the previous section, we now design the numerical scheme for solving the heat equation on a time-evolving domain 
\begin{align}\label{cd-model}
	\frac{\partial u}{\partial t} - \Delta u  = f\quad \text{in}\;\; \Omega_t, \qquad
	u =  0 \quad \text{on}\;\;\Gamma_t, \qquad
	u(\cdot,0) =  u_0  \quad \text{in}\;\;\Omega_0, 
\end{align}	 
where $\Omega_t\subset \bbR^2$ is the time-varying domain, $\Gamma_t = \partial \Omega_t$
the moving boundary defined  in \eqref{eq:Gammat},
$u(\Bx,t)$ the tracer transported by the fluid, $u_0$ the initial value, and $f(\Bx,t)$ the source term distributed in $\bbR^2$ and having a compact support.

By the chain rule and \eqref{def At wn}, the first equation of \eqref{cd-model} can be written as
\begin{equation} \label{cd-eqn1}
	\frac{\D }{\D t} u(\Ca^n(\Bx,t), t) \Big|_{t=t_n} -\Bw^n(\Bx,t_n)\cdot \nabla u(\Bx,t_n) - \Delta u(\Bx,t_n) = f(t_n).
\end{equation}
From \eqref{cd-eqn1}, $\Ca^n(\Bx,t)$ is not necessarily differentiable in $\Bx$. This inspires us to construct a  discrete ALE map in practice.
The semi-discrete scheme is given by
\begin{equation}\label{eq:opr-BDF2}
	\frac{1}{\tau}\Lambda^k_{\Bw} \ul\BU^n  - \Delta u^n = f^n,
\end{equation}
where $u^n$, $f^n$ denote $u(\Bx,t_n)$ and $f(\Bx,t_n)$, respectively, and $\tau^{-1}\Lambda_{\Bw}^k$ stands for the $k^{\text{th}}$-order time finite difference operator, given by 
\begin{equation}\label{eq:opr-BDF}
	\frac{1}{\tau}\Lambda_{\Bw}^k\ul{\BU}^n
	= \frac{1}{\tau }\Lambda^k \ul{\BU}^n 
	-\Bw_k^n(\Bx,t_n)\cdot \nabla U^n\quad \hbox{with}\;\;
	\Lambda^k \ul{\BU}^n = \sum_{i=0}^k \lambda_i^k U^{n-i,n}.
\end{equation} 
Here $\tau^{-1} \Lambda^k$ denotes the BDF-$k$ finite difference operator defined by (cf. \cite{liu13}) and
\begin{equation*}
	\ul \BU=[U^{n-k},\cdots,U^n],\quad 
	U^{n-i} = u(\Ca^n_k(\Bx,t_{n-i}),t_{n-i}),\quad 
	0\leq i\leq k.
\end{equation*}
The coefficients $\lambda_i^k$ for $k=2,3,4$ are listed in Table~\ref{tab:BDF}.

\begin{table}[htbp]
	\caption{\small Coefficients of $\lambda_i^k$ in ${\rm BDF}$ schemes.}\label{tab:BDF} \vspace{-3mm}
	\center 
	\setlength{\tabcolsep}{11mm}
	\begin{tabular}{ c|ccccc}
		\toprule[1pt]
		\diagbox{$k$}{$\lambda_i^k$}{$i$}  
		&$0$      & $1$     &$2$    & $3$   & $4$  \\  \hline   
		$2$      &$3/2 $   & $-2$  & $1/2$ & $0$    & $0$ \\ 
		$3$      &$11/6$   & $-3$  & $3/2$ & $-1/3$ & $0$ \\
		$4$      &$25/12$  & $-4$  & $3$   & $-4/3$ &$1/4$ \\ 
		\bottomrule[1pt]
	\end{tabular}
\end{table}

\subsection{Interface-tracking approximation}

Even the velocity of $\Gamma_t$ is known explicitly, we still need to approximate it with an approximate boundary due to computational complexity in practice. The approximate boundary $\Gamma^n_\eta$ can be constructed by either explicit algorithms such as front-tracking methods \cite{li03} and cubic MARS (Mapping and Adjusting Regular Semi-analytic sets) methods \cite{zha18}, or implicit algorithms such as level set methods \cite{ale99}. The domain enclosed by $\Gamma^n_\eta$ is denoted by $\Omega_\eta^n$ with $\eta$ being the parameter of approximation.

Let $\chibf_n(l)$ and $\hat\chibf_n(l)$, $l\in [0,L]$, denote the parametric representations of $\Gamma_\eta^n$ and $\Gamma^n$, respectively. In order to get optimal error estimates, we make an assumption that the approximation of the boundary is of high order in the sense that
\begin{equation}\label{eq:chi_eta}
	\SN{\chibf_n - \hat\chibf_n}_{C^\mu([0,L])} \lesssim \tau^{k+1-\mu},\quad \mu =0,1.
\end{equation}
For a rigorous proof of the error estimate, we refer to \cite[section~3.2]{ma21} where the cubic MARS method is used for interface tracking \cite{zha18}.
Using \eqref{eq:chi_eta}, we have the error estimate between the exact domain and the approximate domain
\begin{equation*}
	\text{area}(\Omega^n\backslash\Omega_\eta^n)
	+\text{area}(\Omega_\eta^n\backslash\Omega^n)=O(\tau^{k+1}).
\end{equation*}

\begin{remark}
	Assumption \eqref{eq:chi_eta} is used only for analysis. In practice, we require the Hausdorff distance between $\Gamma_\eta^n$ and $\Gamma^n$ to be small, i.e., 
	$\text{dist}(\Gamma_\eta^n,\Gamma^n) \lesssim \tau^{k+1}$.
\end{remark}

\subsection{Finite element spaces}
\label{sec:Finite element spaces and some bilinear forms}

First we define a $\delta$-neighborhood of $\Omega_\eta^n$ by
\begin{equation} \label{eq:delta omega}
	\Omega_{\delta}^n = :\big\{\Bx \in \bbR^2: \min_{\By\in \Omega_\eta^n}|\Bx -\By| \leq 0.5\tau\big\}.
\end{equation}
It is easy to see $\Omega_\eta^n \subset \Omega_\delta^n$ (see Fig.~\ref{fig:mesh}).
Let $D$ be an open square satisfying $\Omega_t\cup\Omega^n_\delta \subset D$ 
for all $0\le t\le T$ and $0\le n\le N$.
Let $\Ct_h$ be the uniform partition of $\bar D$ into 
closed squares of side-length $h$. 
It generates two covers of $\Omega^n_{\delta}$ and $\Gamma^n_{\eta}\cup \partial \Omega_\delta^n$, respectively,  
\begin{align}\label{eq:Cth}
	\Ct^n_h &:= \left\{K\in\Ct_h:\; 
	\ol K\cap\Omega^n_{\delta}\neq \emptyset \right\}, \\
	\Ct^n_{h,B} &:= \left\{K\in \Ct^n_h:\; 
	\ol K\cap\Gamma^n_{\eta}\neq \emptyset \;\;\text{or}\;\;
	\ol K\cap\partial \Omega_\delta^n\neq \emptyset \right\} . 
\end{align}
The cover $\Ct^n_h$ generates a fictitious domain
$\tilde\Omega^n:= \mathrm{interior}\big(\cup_{K\in\Ct^n_h}K\big)$.
Let $\Ce_h$ be the set of all edges in $\Ct_h$. The set of interior edges of boundary elements are denoted by
\begin{equation}\label{eq:Ceb}
	\Ce_{h,B}^{n}= \big\{E\in\Ce_h: \; 
	\exists K\in \Ct^n_{h,B}\;\;
	\hbox{s.t.}\;\; E\subset\partial K
	\backslash \partial \tilde\Omega^n\big\}.
\end{equation}

\begin{figure}[t]
	\centering
	\begin{tikzpicture}[scale =1.5]
		\filldraw[red](0.2*2,0.2*2)--(0.2*8,0.2*2)--(0.2*8,0.2*3)--(0.2*9,0.2*3)--(0.2*9,0.2*7)
		--(0.2*8,0.2*7)--(0.2*8,0.2*8)--(0.2*2,0.2*8)--(0.2*2,0.2*7)--(0.2*1,0.2*7)--(0.2*1,0.2*3)
		--(0.2*2,0.2*3)--(0.2*2,0.2*2);
		
		\filldraw[white](0.2*3,0.2*6)--(0.2*4,0.2*6)--(0.2*4,0.2*7)--(0.2*5,0.2*7)--(0.2*6,0.2*7)--(0.2*6,0.2*6)
		--(0.2*7,0.2*6)--(0.2*7,0.2*5)--(0.2*7,0.2*4)--(0.2*6,0.2*4)--(0.2*6,0.2*3)--(0.2*5,0.2*3)--(0.2*4,0.2*3)
		--(0.2*4,0.2*4)--(0.2*3,0.2*4)--(0.2*3,0.2*5)--(0.2*3,0.2*6); 
		
		\filldraw[yellow](0.2*3,0.2*6)--(0.2*4,0.2*6)--(0.2*4,0.2*7)--(0.2*5,0.2*7)--(0.2*6,0.2*7)--(0.2*6,0.2*6)
		--(0.2*7,0.2*6)--(0.2*7,0.2*5)--(0.2*7,0.2*4)--(0.2*6,0.2*4)--(0.2*6,0.2*3)--(0.2*5,0.2*3)--(0.2*4,0.2*3)
		--(0.2*4,0.2*4)--(0.2*3,0.2*4)--(0.2*3,0.2*5)--(0.2*3,0.2*6);

		\draw [step =0.2cm,gray,thin] (0,0) grid (2cm,2cm);
		\draw[black, thick] (1,1) ellipse [x radius=0.56cm, y radius=0.5cm];
		\draw[dashed,green,ultra thick] (1,1) ellipse [x radius=0.64cm, y radius=0.56cm];
		
		\draw[red,thick](0.2*2,0.2*2)--(0.2*8,0.2*2)--(0.2*8,0.2*3)--(0.2*9,0.2*3)--(0.2*9,0.2*7)
		--(0.2*8,0.2*7)--(0.2*8,0.2*8)--(0.2*2,0.2*8)--(0.2*2,0.2*7)--(0.2*1,0.2*7)--(0.2*1,0.2*3)
		--(0.2*2,0.2*3)--(0.2*2,0.2*2);
		
		\node[left] at (0.3,1.7) {$D$};
		\node[left] at (1.2,0.2*5) {$\Omega_\eta^n$};
		\node[right] at (1.4,0.2*7.35){$\Omega_{\delta}^n$};
	\end{tikzpicture}
	\qquad
	\begin{tikzpicture}[scale =1.8]
		\filldraw[red](0.2*2,0.2*2)--(0.2*8,0.2*2)--(0.2*8,0.2*3)--(0.2*9,0.2*3)--(0.2*9,0.2*7)
		--(0.2*8,0.2*7)--(0.2*8,0.2*8)--(0.2*2,0.2*8)--(0.2*2,0.2*7)--(0.2*1,0.2*7)--(0.2*1,0.2*3)
		--(0.2*2,0.2*3)--(0.2*2,0.2*2);			
		\filldraw[yellow](0.2*3,0.2*6)--(0.2*4,0.2*6)--(0.2*4,0.2*7)--(0.2*5,0.2*7)--(0.2*6,0.2*7)--(0.2*6,0.2*6)
		--(0.2*7,0.2*6)--(0.2*7,0.2*5)--(0.2*7,0.2*4)--(0.2*6,0.2*4)--(0.2*6,0.2*3)--(0.2*5,0.2*3)--(0.2*4,0.2*3)
		--(0.2*4,0.2*4)--(0.2*3,0.2*4)--(0.2*3,0.2*5)--(0.2*3,0.2*6);			
		\draw[red,thick](0.2*2,0.2*2)--(0.2*8,0.2*2)--(0.2*8,0.2*3)--(0.2*9,0.2*3)--(0.2*9,0.2*7)
		--(0.2*8,0.2*7)--(0.2*8,0.2*8)--(0.2*2,0.2*8)--(0.2*2,0.2*7)--(0.2*1,0.2*7)--(0.2*1,0.2*3)
		--(0.2*2,0.2*3)--(0.2*2,0.2*2);			
		\draw[black, ultra thick] (1,1) ellipse [x radius=0.56cm, y radius=0.5cm];
		\draw[step =0.2cm,gray,thin] (0.6,0.6) grid (7*0.2cm,7*0.2cm);
		\draw[blue,ultra thick](0.2*2,0.2*3)--(0.2*2,0.2*7);
		\draw[blue,ultra thick](0.2*3,0.2*2)--(0.2*3,0.2*8);
		\draw[blue,ultra thick](0.2*4,0.2*2)--(0.2*4,0.2*4);
		\draw[blue,ultra thick](0.2*4,0.2*6)--(0.2*4,0.2*8);
		\draw[blue,ultra thick](0.2*5,0.2*2)--(0.2*5,0.2*3); 
		\draw[blue,ultra thick](0.2*5,0.2*7)--(0.2*5,0.2*8);
		\draw[blue,ultra thick](0.2*6,0.2*2)--(0.2*6,0.2*4); 
		\draw[blue,ultra thick](0.2*6,0.2*6)--(0.2*6,0.2*8);
		\draw[blue,ultra thick](0.2*7,0.2*2)--(0.2*7,0.2*8);
		\draw[blue,ultra thick](0.2*8,0.2*3)--(0.2*8,0.2*7);
		
		\draw[blue,ultra thick](0.2*2,0.2*3)--(0.2*8,0.2*3);
		\draw[blue,ultra thick](0.2*1,0.2*4)--(0.2*4,0.2*4); 
		\draw[blue,ultra thick](0.2*6,0.2*4)--(0.2*9,0.2*4);
		\draw[blue,ultra thick](0.2*1,0.2*5)--(0.2*3,0.2*5); 
		\draw[blue,ultra thick](0.2*7,0.2*5)--(0.2*9,0.2*5);
		\draw[blue,ultra thick](0.2*1,0.2*6)--(0.2*4,0.2*6); 
		\draw[blue,ultra thick](0.2*6,0.2*6)--(0.2*9,0.2*6);
		\draw[blue,ultra thick](0.2*2,0.2*7)--(0.2*8,0.2*7);			
		\node[left] at (2.15,1.0) {$\tilde{\Omega}^n$};
		\node[right] at (1.25,1.5) {$\Ct_{h,B}^n$};
	\end{tikzpicture}
	\caption{\small Left: the domain $D$ and its partition $\Ct_h$, the approximate domain $\Omega_\eta^n$ and its $\delta$-neighborhood $\Omega_{\delta}^n$.
		Right: the set of red and yellow squares $\Ct^n_h$, the set of red squares $\Ct^n_{h,B}$, the set of blue edges $\Ce_{h,B}^{n}$, the union of red and yellow squares $\tilde{\Omega}^n$.} \label{fig:mesh}
\end{figure}
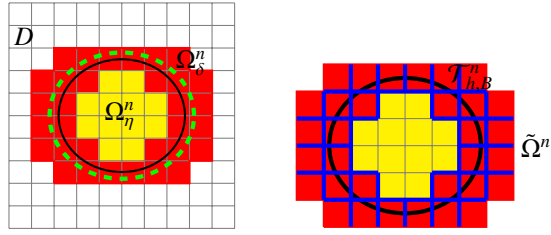

Next we define the finite element spaces on $D$ and $\tilde \Omega^n$, respectively, as follows
\begin{align*}
	V(k,\Ct_h) :=\,& \big\{v\in\Hone[D]:
	v|_K\in Q_k(K),\;\forall\,K\in \Ct_h\big\}, \\
	V(k,\Ct^n_h):=\,&  \big\{v|_{\tilde{\Omega}^n}:
	v\in V(k,\Ct_h) \big\} ,
\end{align*}
where $Q_k$ is the space of polynomials whose degrees are no more than $k$ for each variable. 
The space of piecewise regular functions over $\Ct^n_h$ is defined by
\begin{equation}\label{eq:FEM space}
	H^m(\Ct^n_h) := \big\{v\in\Ltwo[\tilde{\Omega}^n]:\;
	v|_K\in H^m(K),\;\forall\, K\in\Ct^n_h\big\},\qquad m\ge 1.
\end{equation}

\subsection{Construction of an approximate ALE map}\label{sec:construction ALE}
Let $\BX_\tau^{n,n-1}$ denote the approximation of 
the inverse of $\BX_F(t_n,t_{n-1},\Bx)$. 
This is computed by solving \eqref{eq:X} with the RK-$(k+1)$ scheme (the $(k+1)^{\text{th}}$-order Runge-Kutta scheme) from $t_n$ to $t_{n-1}$. 
For any $\Bx^{n}\in \Gamma_\eta^n$, the point 
$\Bx^{n-1}=\BX_\tau^{n,n-1}(\Bx^n)$ is calculated 
as follows:
\begin{equation}\label{eq:RKrk}
	\begin{cases}
		\Bx^{(1)} = \Bx^{n},  \\
		\Bx^{(i)}  =\Bx^{n} -
		\tau \sum \limits_{j=1}^{i-1} 
		a^{k+1}_{ij} \Bv(\Bx^{(j)},t_{n} - c_j^{k+1} \tau),\quad
		2\le i\le n_{k+1},  
		\vspace{0.5mm}\\
		\Bx^{n-1} =\Bx^{n} -
		\tau \sum \limits_{i=1}^{n_{k+1}} 
		d_{i}^{k+1} \Bv(\Bx^{(i)},t^{(i)}) .
	\end{cases}
\end{equation}
Here $a_{ij}^{k+1},\;d_i^{k+1},\;c_i^{k+1}$ are the coefficients of the RK-($k+1$) scheme, 
satisfying $a_{ij}^{k+1} =0$ if $j\geq i$, and $n_{k+1}$ is the number of stages. 
Recall from section~\ref{sec:backmap} that $\Bg^{n,n-1}$ maps $\Gamma^n$ to $\Gamma^{n-1}$. We define its approximation by
$\Bg_{\eta}^{n,n-1}= \BX_{\tau}^{n,n-1}|_{\Gamma_{\eta}^n}$.

A  discrete approximation of $\BX^{n,n-1}$ is defined by solving the discrete problem: find $\BX_h^{n,n-1}\in \BV(k,\Ct_h^n) :=(V(k,\Ct_h^n))^2$ such that
\begin{equation}\label{eq:discrte Xh} 
	\mathscr{A}_h^n(\BX_h^{n,n-1},\Bv_h) =
	\Cf_{\Gamma_{\eta}^n}(\Bg_{\eta}^{n,n-1},\Bv_h)
	\quad\; \forall\, \Bv_h \in \BV(k,\Ct_h^n).
\end{equation}
where $\Cf_{\Gamma_{\eta}^n}(\Bw,\Bv_h) = \left<\Bw,\gamma_0h^{-1}\Bv_h-\partial_{\Bn} \Bv_h\right>_{\Gamma_\eta^n}$ and the bilinear forms are defined by 
\begin{align}
	\bm{\mathscr{A}}_h^n(\Bu,\Bv) :=\,& \mathscr{A}_h^n(u_1,v_1)+\mathscr{A}_h^n(u_2,v_2), \label{A-nh}\\
	\mathscr{A}^n_h(u,v):=\,&(\nabla u, \nabla v)_{\Omega^n_{\eta}} + \mathscr{S}^{n}_h(u,v) 
	+ \mathscr{J}_{0}^{n}(u,v) +\mathscr{J}_{1}^{n}(u,v) , \label{a-nh}\\	
	\mathscr{S}^n_h(u,v):=\,&- \left<v,\partial_\Bn u\right>_{\Gamma_{\eta}^n}
	-\left<u,\partial_\Bn v\right>_{\Gamma_{\eta}^n},\label{S-nh} \\
	\mathscr{J}^n_0(u,v):=\,& \frac{\gamma_0}{h}\left<u, v\right>_{\Gamma_{\eta}^n},\quad 
	\mathscr{J}^n_1(u,v):= \sum_{E\in \Ce_{h,B}^{n}} 
	\sum_{l=1}^k h^{2l-1}\left<\jump{\partial_{\Bn}^l u},
	\jump{\partial_{\Bn}^l v} \right>_{E}. \label{J1-nh}
\end{align}
In \eqref{S-nh}, $\partial_{\Bn} u=\frac{\partial u}{\partial \Bn}$ denotes the normal derivative of $u$ 
on $\Gamma_{\eta}^n$. In \eqref{J1-nh}, $\mathscr{J}^n_0$ is used to impose the Dirichlet boundary condition weakly with $\gamma_0>0$, $\partial_{\Bn}^l v$ denotes 
the $l$-th order normal derivative of $v$, and $\jump{v}|_E = v|_{K_1}-v|_{K_2}$ denotes the jump of $v$ across edge $E$ where $K_1,K_2\in\Ct_h$ are the two elements sharing $E$.

Similar to \eqref{def At k}, we define the fully discrete ALE map and artificial velocity by
\begin{equation}\label{eq: ALE Ah}
	\Ca^{n}_{k,h}(\Bx,t) = \sum_{i=0}^k l_n^i(t)\BX_h^{n,n-i},\quad
	\Bw_{k,h}^{n}(\Bx,t) = \sum_{i=0}^k (l_n^i)'(t) \BX_h^{n,n-i},
\end{equation}
where $\BX_h^{n,n}=\BI$ and the multi-step backward flow map is defined by
\begin{align}\label{def:Xh n-i}
	\BX_h^{n,n-i} := \BX_h^{n-1,n-i}\circ \BX_h^{n,n-1}
	=\BX_h^{n-i+1,n-i}\circ \cdots\circ \BX_h^{n,n-1}.
\end{align}
The recursive definition only needs to compute the one-step map  $\BX_h^{n,n-1}$ at each $t_n$.

\subsection{The fully discrete scheme}
Given the finite element function $\Bw_{k,h}^{n}$, 
we define the $k^{\text{th}}$-order time finite difference operator as
\begin{equation}\label{eq:time diff wh}
	\frac{1}{\tau} \Lambda^k_{\Bw_h} \ul{\BU^n_h}
	= \frac{1}{\tau}\Lambda^k \ul{\BU^n_h} -\Bw_{k,h}^n\cdot \nabla U^{n,n}_h,
\end{equation}
where $\ul{\BU^n_h}=\big[U^{n-k,n}_h,\cdots,U^{n,n}_h\big]^\top$ and  $U_h^{n-i,n}=u_h^{n-i}\circ\BX_h^{n,n-i}$. Note that $U_h^{n,n} = u_h^n$. 
The discrete approximation to problem \eqref{cd-model} is to 
seek  $u^n_h\in V(k,\Ct_h^{n})$ such that 
\begin{align}	\label{eq:discrete-uh}
	\frac{1}{\tau}\big(\Lambda^k_{\Bw_h} \ul{\BU^n_h}, v_h\big)_{\Omega^n_{\eta}}
	+ \mathscr{A}^n_h(u_h^n,v_h) 
	= (f^n,v_h)_{\Omega^n_{\eta}} \quad
	\forall\,v_h\in V(k,\Ct_h^{n}).
\end{align}	
In view of \eqref{eq:discrete-uh}, the stiffness matrix corresponding to $\mathscr{A}_h^n$ has already been obtained when computing the discrete ALE map in \eqref{eq:discrte Xh}. The $\delta$-neighborhood $\Omega_{\delta}^n$ is chosen to ensure that $u_h^{n-i}\circ \BX_h^{n,n-i}$ is well-defined for $0\leq i\leq k$.
\subsection{Well-posedness and error estimates}
\label{sec:The analysis for the fully scheme}
In this section, we show the well-posedness, stability, and error estimates of the discrete problems in the appendix. 
Firstly, we define the mesh-dependent norms
\begin{align*}
	&\TN{v}_{\Omega^n_{\eta}}  = \big(\SNHone[\Omega^n_{\eta}]{v}^2 
	+h^{-1}\NLtwo[\Gamma^n_{\eta}]{v}^2 
	+ h\NLtwo[\Gamma^n_{\eta}]{\partial_{\Bn}v}^2\big)^{1/2} ,\\
	&\TN{v}_{\Ct^n_h}  = \big(\SNHone[\Omega^n_{\eta}]{v}^2 
	+\mathscr{J}^n_0(v,v) +\mathscr{J}^n_1(v,v)\big)^{1/2} .
\end{align*}
Clearly $\TN{\cdot}_{\Ct^n_h}$ is a norm on $H^1(\tilde \Omega^{n})\cap H^{k+1}(\Ct^n_h)$. From \cite{guz18}, we have the following norm inequalities: for any $v_h\in V(k,\Ct_h^{n})$, 
\begin{align}	\label{norm-eq0}
	\NLtwo[\tilde \Omega^n]{v_h}^2
	\lesssim \NLtwo[\Omega^n_{\eta}]{v_h}^2 
	+h^2\mathscr{J}_1^n(v_h,v_h), 	\qquad
	\TN{v_h}_{\Omega^n_{\eta}} \lesssim \TN{v_h}_{\Ct^n_h}.
\end{align}
Suppose $\gamma_0$ is large enough. It is standard to prove the coercivity and continuity of the bilinear form (see \cite{ma21}): 
for any $v_h\in V(k,\Ct^n_h)$ and $u\in H^{k+1}(\Ct_h^n)\cap H^1(\Omega_{\eta}^n)$, 
\begin{align*}
	\mathscr{A}_h^n(v_h,v_h) \geq C_a\TN{v_h}_{\Ct^n_h}^2, \qquad
	\SN{\mathscr{A}_h^n (u,v_h)} \lesssim 
	\TN{u}_{\Omega_{\eta}^n}\TN{v_h}_{\Ct^n_h}.
\end{align*}
By Corollary~\ref{corol:bound of Jh}, the artificial velocity $\Bw_{k,h}^n$ is bounded. This gives the theorem.

\begin{theorem}\label{lem:Ah}
	There is a positive constant $\tau_0 \lesssim C_a\|\Bw_{k,h}^n\|_{\BL^{\infty}(\Omega_{\eta}^n)}^{-2}$ small enough such that the problem \eqref{eq:discrete-uh} has a unique solution for any $\tau \in (0, \tau_0]$.
\end{theorem}

\begin{theorem} \label{thm:uh-stab}
	Assume the approximate boundary satisfies \eqref{eq:chi_eta} and  the penalty parameter $\gamma_0$ in $\mathscr{A}^n_h$ is large enough.
	Let $u_h^{n}$ be the solution to the discrete problem \eqref{eq:discrete-uh} 
	based on the pre-calculated initial values $\{u_h^{0},\cdots, u^{k-1}_h\}$. 
	There is an $h_0>0$ small enough such that, 
	for any $h=O(\tau)\in(0, h_0]$ and $m\ge k$, 
	\begin{align*}
		\|{u^m_h}\|^2_{L^2(\Omega^n_{\eta})}
		+\sum_{n=k}^m\tau\TN{u_h^n}^2_{\Ct_h^n} 
		\lesssim\, \sum_{n=k}^m
		\tau\|{f^n}\|^2_{L^2(\Omega^n_\eta)} 
		+\sum_{i=0}^{k-1}\left[
		\|{u^i_h}\|^2_{L^2(\Omega^i_\eta)} +
		\tau \|u^{i}_h\|_{H^1(\tilde\Omega^{i})}^2
		\right]. 
	\end{align*}
\end{theorem}
The proof of Theorem~\ref{thm:uh-stab} is presented in \ref{sec:The proof of the stability for the discrete solutions}.

For $v\in W^{\mu,\infty}(0,T;H^s(\Omega_t))$, one has $\|v(t)\|_{H^s(\Omega_t)}\in W^{\mu,\infty}([0,T])$
for $\mu =0,1$ and $s=1,\cdots, k+1$.
Suppose $f\in L^{\infty}(0,T;H^1(D))$. The exact solution satisfies
\begin{equation}\label{ass exact solution}
	u\in H^{k+1}(Q_T) \cap L^{\infty}(0,T;H^{k+1}(\Omega_t))\cap W^{1,\infty}(0,T;H^1(\Omega_t)), 
\end{equation} 
where $Q_T = \{ (\Bx,t):\Bx\in \Omega_t, t\in [0,T]\}$.
Assume $\Ca^n$ is smooth in time such that 
\begin{equation}\label{assum ALE mapping}
	\| \partial_t^{k+1}\Ca^n(\Bx,t)\|_{L^2(\Omega_t)} \leq M_0
	\quad  \forall\, t\in [t_{n-k}, t_n],
\end{equation}
where $k$ is the order of the BDF scheme and $M_0$ is independent of $\tau$, $\eta$, $h$ and $n$. 
Since $\Ca^{n}_k$ is the $k^{\text{th}}$-order Lagrange interpolation of $\Ca^n$ in $[t_{n-k},t_n]$, 
we have
\begin{equation}\label{eq:bound of A}
	\sum_{i=0}^k\|\partial_t^i\Ca^{n}_{k,h}\|_{L^2(\Omega_{\eta}^n)} \lesssim \sum_{i=0}^k\|\partial_t^i\Ca^{n}_k\|_{L^2(\Omega_t)} \lesssim \|\partial_t^{k+1}\Ca^n\|_{L^2(\Omega_t)}
	\quad \forall\, t \in [t_{n-k},t_n].
\end{equation}

\begin{theorem}\label{thm:err}
	Suppose the assumptions in Theorem~\ref{thm:uh-stab} hold and the initial solutions satisfy $\|{u^i-u_h^i}\|^2_{L^2(\Omega_\eta^i)}
	+\tau \||{u^i-u_h^i}|\|_{\Omega_\eta^i}^2
	\lesssim \tau^{2k}$ for $\le i\le k-1$.	Then 
	\begin{equation}\label{um-err}
		\|{u^m-u_h^m}\|^2_{L^2(\Omega^m_\eta)}
		+ \sum_{n=k}^{m}\tau \||{u^{n}-u_h^{n}}|\|_{\Ct_h^n}^2 
		\lesssim \tau^{2k},\qquad
		k\leq m\leq N.
	\end{equation} 
\end{theorem}

By \cite{ma21}, there is an extension $\tilde u$ of $u$
such that, for $1\leq m\leq k+1$,
\begin{equation}\label{eq:property of extension}
	\begin{cases} 
		\|\tilde u(t)\|_{H^{m}(D)}\lesssim \|u(t)\|_{H^{m}(\Omega_t)}, \vspace{0.5mm}\\
		\|\tilde u(t)\|_{H^{k+1}(D\times[0,T])} \lesssim \|u(t)\|_{H^{k+1}(Q_T)},
		\vspace{0.5mm}\\
		\|\partial_t \tilde u(t)\|_{H^1(D)}\lesssim \|u(t)\|_{H^2(\Omega_t)} + 
		\|\partial_t u(t)\|_{H^1(\Omega_t)}.
	\end{cases}
\end{equation}
Then \eqref{eq:bound of A} and \eqref{eq:property of extension} show $\big\|\frac{\D^{k+1}}{\D t^{k+1}} \tilde u(\Ca_h^n(x,t),t_n)\big\|_{L^{2}(\Omega_{\eta}^n)} \lesssim M_0 \|u\|_{H^{k+1}(\Omega_t\times[t_{n-k},t_n] )}$. It is easy to see that the extension $\tilde{u}$ satisfies
\begin{align}
	\frac{1}{\tau}(\Lambda^k_{w_h} \tilde{\ul\BU},v_h)_{\Omega_{\eta}^n}
	+\mathscr{A}_h^n(\tilde u,v_h) 
	=\left<\tilde u^n,\,
	(\gamma_0/h-\partial_n) v_h\right>_{\Gamma_\eta^n}
	+(\tilde f^n+R^n,v_h)_{\Omega_{\eta}^n},
	\label{eq:tild u}
\end{align}
where $\tilde{\ul\BU}=[\tilde u^{n-k}\circ \Ca_{k,h}^n(\cdot,t_{n-k}),\cdots, \tilde u^n]$, $\tilde f^n = \partial_t \tilde u(t_n)-\Delta \tilde u^n$, and
\begin{align*} 
	R^n = \tau^{-1}\Lambda^k \tilde{\ul\BU}-
	\frac{\D}{\D t}\tilde u\circ \Ca_{k,h}^n(\cdot,t)|_{t=t_n}.
\end{align*}
The proof of Theorem~\ref{thm:err} is parallel to \cite[Section~6]{ma21} by subtracting \eqref{eq:discrete-uh} from \eqref{eq:tild u}. We omit the details here.

\subsection{The ALE-UFE framework}
Now we conclude the ALE-UFE framework for solving PDEs on varying domains. It consists of four steps. 
\vspace{1mm}
\begin{center}
	\fbox{\parbox{0.975\textwidth}
		{\begin{enumerate}[leftmargin=5mm]
				\item Track the varying interface by the forward boundary map \eqref{eq:X}.
				\vspace{0.5mm}
				
				\item  Construct a one-step backward flow map with \eqref{eq:discrte Xh}, a  
				discrete ALE map and an  artificial velocity with \eqref{eq: ALE Ah}, and a multi-step map $\BX^{n,n-i}_h$ with \eqref{def:Xh n-i}. 
				\vspace{0.5mm}
				
				\item Define the finite difference operator in \eqref{eq:time diff wh} by combining the BDF-$k$ scheme and the backward flow map.
				\vspace{0.5mm}
				
				\item Construct the fully discrete scheme as in \eqref{eq:discrete-uh}. 
	\end{enumerate}}}
\end{center}
\vspace{1mm}

The framework can be conveniently  applied to various moving-domain problems. The operations in each step  are adapted to a specific problem.

\section{A domain-PDE-coupled problem}
\label{sec:nonlinear}

In this section we apply the ALE-UFE framework to a nonlinear problem where the moving domain depends on the solution. Unless otherwise specified, the finite element spaces and bilinear forms are defined in the same way as in the previous section. We consider the following model:
\begin{equation}\label{eq:nonlienar}
	\frac{\partial \Bv}{\partial t} -\Delta \Bv= \Bf\quad\text{in}\;\; \Omega_t, \quad
	\partial_{\Bn} \Bv =\Bg_{N}\quad\text{on} \; \Gamma_t,	\qquad
	\Bv(\Bx,t_0) = \Bv_0,\;\; \text{in}\;\Omega_0, 
\end{equation}
where  $\Omega_t$ is the domain surrounded by a varying boundary $\Gamma_t$, i.e. $\Gamma_t = \partial \Omega_t$.
The Neumann boundary condition $\Bg_N$ can be viewed as an applied force on $\Gamma_t$, such as a surface tension of fluids \cite{ma23}.
For simplicity, we treat it as a given function.

\subsection{Flow maps}
The variation of $\Gamma_t$ has the form of \eqref{eq:Gammat}, namely
$\Gamma_t = \BX_F(t;0,\Gamma_0)$, and the forward boundary map is defined by
\begin{equation} \label{eq:nonlinear bd}
	\frac{\D}{\D t} \BX_F(t;s,\Bx_s) = \Bv(\BX_F(t;s,\Bx_s),t),\quad
	\BX_F(s;s,\Bx_s) = \Bx_s.
\end{equation}
Note that $\Bv$ is the solution of the problem.
Based on an ALE map $\Ca^n(\Bx,t):\Omega^n\to \Omega_t$, $(t\leq t_n)$, the equation in \eqref{eq:nonlienar} can be written as follows
\begin{equation}\label{eq:nonlinear compact}
	\frac{\D\, \Bv(\Ca^n(\Bx,t),t)}{\D t} \Big|_{t=t_n} =
	\Bw^n(\Bx,t_n) \nabla \Bv(\Bx,t_n) + \Delta \Bv(\Bx,t_n)+\Bf(\Bx,t_n),
\end{equation}
where $\Bw^n(\Bx,t_n) = \partial_t \Ca^n(\Bx,t)|_{t=t_n}$.
Here, the unknown variables are the flow velocity
$\Bv$ and the moving boundary $\Gamma_t$.
Equation \eqref{eq:nonlinear bd} governs the evolution of the boundary through the forward boundary map and Equation \eqref{eq:nonlinear compact} governs the dynamics of the fluid through an ALE map.
They form a nonlinear system. We adopt semi-implicit BDF schemes for solving them:
$\BX_F$ is computed explicitly at each time step, while the update of $\Bv$
is done implicitly.
The semi-discrete scheme reads as follows,
\begin{align}
	&\frac{1}{\tau}\sum_{i=0}^k a_i^k \BX_F^{n-1,n-i}(\Bx) = \sum_{i=1}^k b_i^k\Bv^{n-i}\circ \BX_F^{n-1,n-i}(\Bx) \quad \forall \Bx \in \Gamma^{n-1}, \label{eq:semi bd}\\
	&\frac{1}{\tau}\sum_{i=0}^k a_i^k \Bv^{n-i}\circ \Ca_k^n(\cdot,t_{n-i})
	=  \Bw_k^n(\Bx,t_n)\nabla \Bv^n+ \Delta \Bv^n+\Bf^n,
	\label{eq:semi equa}
\end{align}
where the coefficients $a_i^k$, $b_i^k$ are listed in Table~\ref{tab:SBDF} (cf. \cite{asc95}), $\Bv^{n-i} = \Bv(\cdot,t_{n-i})$, and $\Bf^n=\Bf(\cdot,t_{n})$.
Here $\BX_F^{n-1,n-i} = (\BX_F^{n-i,n-1})^{-1}$ is the inverse of $\BX_F^{n-i,n-1}$.
From \eqref{eq:semi bd}, an explicit forward map on the boundary is defined by
\begin{equation*}
	\BX_F^{n-1,n}(\Bx) = \frac{1}{a_0^k}\sum_{i=1}^k
	\big[\tau b_i^k\Bv^{n-i}\circ \BX_F^{n-1,n-i}(\Bx)- a_i^k \BX_F^{n-1,n-i}(\Bx)\big].
\end{equation*}

\renewcommand\arraystretch{1.1}
\begin{table}[htbp]
	\caption{Coefficients for $a_i^k,b_i^k$ schemes.}\label{tab:SBDF}\vspace{-3mm}
	\center 
	\setlength{\tabcolsep}{8.1mm}
	\begin{tabular}{ c|lllll}
		\toprule[1pt]
		\diagbox{$k$}{$(a_i^k,b_i^k)$}{$i$}  &$0$     & $1$     &$2$    & $3$   & $4$  \\  \hline
		$2$      &$(3/2,\cdotbf)$   & $(-2,2)$  & $(1/2,-1)$ & $(0,0)$    & $(0,0)$\\ 
		$3$      &$({11}/{6},\cdotbf)$   & $(-3,3)$  & $(3/2,-3)$ & $(-1/3,1)$ & $(0,0)$ \\ 
		$4$      &$({25}/{12},\cdotbf)$  & $(-4,4)$  & $(3,-6)$   & $(-4/3,4)$ &$(1/4,-1)$ \\ 
		\bottomrule[1pt]
	\end{tabular}
\end{table}

\subsection{Surface tracking algorithm via forward boundary map}
In practice, we construct a discrete approximation $\BX_{\Cf_h}^{n-1,n}$ 
to $\BX_F^{n-1,n}$ for surface tracking.
Suppose that the approximate boundary $\Gamma_{\eta}^{n-i}$, 
the discrete solutions $\Bv^{n-i}_h\in \BV(k,\Ct_h^{n-i})$,
and the approximate backward boundary maps 
$\BX_{\Cf_h}^{n-1,n-i}:\Gamma_\eta^{n-1}\to \Gamma_\eta^{n-i}$ have been obtained for $1\le i<k$. We use the information to construct the forward boundary map 
$\BX_{\Cf_h}^{n-1,n}:\Gamma_\eta^{n-1}\to \bbR^2$, explicitly, as follows
\begin{equation}\label{eq:nonlinear forward}
	\BX_{\Cf_h}^{n-1,n}(\Bx)=\frac{1}{a_0^k}\sum_{i=1}^k
	\big[\tau b_i^k\Bv_h^{n-i}\circ \BX_{\Cf_h}^{n-1,n-i}(\Bx)- a_i^k \BX_{\Cf_h}^{n-1,n-i}(\Bx)\big]\quad
	\forall \Bx \in \Gamma_{\eta}^{n-1}.
\end{equation}

Let $\Cp^0=\left\{\Bp^0_j: 0\le j\le J^0\right\}$ be the set of control points on the initial boundary  $\Gamma^0_\eta:=\Gamma_0$. Suppose that the arc length of $\Gamma^0_\eta$ between $\Bp^0_0$ and $\Bp^0_{j}$ equals to $L^0_j=j\eta$ for $1\le j\le J^0$, where $\eta:= L^0/J^0$ and $L^0$ is the arc length of $\Gamma^0_\eta$. For all $0\le m<n$, suppose we are given with the set of control points $\Cp^m=\{\Bp^m_j: 0\le j\le J^m\}\subset\Gamma^m_\eta$ and the parametric representation $\chibf_m$ of $\Gamma^m_\eta$, which satisfies
\begin{equation*}
	\chibf_m(L^m_j) = \Bp^m_j, \quad
	L^m_j = \sum_{i=0}^j\SN{\Bp^m_{i+1}-\Bp^m_{i}},\quad 0\le j\le J^m.
\end{equation*}
The set $ \Cp^n=\{\Bp_j^n: j=0,\cdots,J^n\}$ are obtained by 
\eqref{eq:nonlinear forward} and satisfy 
$|\Bp_{j+1}^n-\Bp_j^n|\sim \eta$.
We adopt the surface-tracking method in \cite[Algorithm~3.1 ]{maz21} to
construct a $C^2$ smooth boundary $\Gamma_{\eta}^n$ with cubic spline interpolation.

\subsection{Construction of an approximate ALE map}
Since ALE maps depend on boundary motions, the first task is to construct an approximation of the inverse map  $(\BX_{\Cf_h}^{n-1,n})^{-1}$, denoted by $\BX_{\Cf_h}^{n,n-1}$ without causing confusions. 
The backward flow map $\BX_{\Cf_h}^{n,n-1}$ in \eqref{eq:nonlinear forward} will also be used to construct $\Gamma_{\eta}^{n+1}$ at the next time step. The construction of $\BX_{\Cf_h}^{n,n-1}$ consists of two steps. 

{\bf Step~1. Construct an approximation of $\BX^{n-1,n}_{\Cf_h}$.} 
Let $\Gamma^{n-1}_{\eta,K} = \Gamma^{n-1}_{\eta}\cap K$ and let
$M$ be the number of control points in the interior of $K$,
which divide the curve into $M+1$ segments, namely,
\begin{equation}
\Gamma^{n-1}_{\eta,K} = \Gamma_{K,0}^{n-1} \cup \Gamma_{K,1}^{n-1}\cup \cdots \cup \Gamma_{K,M}^{n-1},
\quad \mathring{\Gamma}_{K,l}^{n-1}\cap \mathring{\Gamma}_{K,m}^{n-1} =\emptyset, 
\quad l\neq m.
\end{equation}
For each $\Gamma^{n-1}_{K,m}$, we take $k+1$ nodal points $\BA_0,\cdots,\BA_k$ quasi-uniformly on $\Gamma^{n-1}_{K,m}$ with $\BA_0,\BA_k\in \Cp^{n-1}$ (see Fig.~\ref{fig:cut element and illustration}(a)). Let $\hat{I}=[0,1]$ be the reference interval.
The two isoparametric transforms are defined as
\begin{align}\label{FmGm-T}
\BF_{{K,m}}(\xi) := \sum_{i=0}^k
\BA_{i}b_{i}(\xi),\quad
\BG_{{K,m}}(\xi) := \sum_{i=0}^k \BA_i^n b_{i}(\xi)\quad
\forall\,\xi\in\hat{I},
\end{align}
where $b_{i}\in P_k(\hat{I})$ satisfies $b_{i}(l/k)=\delta_{i,l}$ and $\BA_i^n = \BX_{\Cf_h}^{n-1,n}
\big(\BA_{i}\big)$.
They define a homeomorphism from 
$\Gamma_{K,m}^{n-1}$ to $\tilde{\Gamma}_{K,m}^n :=\big\{G_{{K,m}}(\xi):\xi\in\hat{I}\big\}$:
\begin{equation}\label{Kmn}
\tilde\BX^{n-1,n}_{{K,m}} :=G_{{K,m}}\circ F_{{K,m}}^{-1}.
\end{equation}
Define $\tilde{\Gamma}_K^n:=\cup_{m=0}^M \tilde{\Gamma}_{K,m}^n$, and 
$\tilde{\Gamma}^n_{\eta}:=\cup_{K\in\Ct^{n-1}_{h,B}}\tilde{\Gamma}_K^n$. 
We obtain a homeomorphism $\tilde\BX^{n-1,n}_{\Cf_h}$: $\Gamma^{n-1}_\eta\to \tilde{\Gamma}^n_{\eta}$ which is defined piecewisely as follows (see Fig.~\ref{fig:cut element and illustration}(b))
\begin{equation}\label{tXn}
\tilde\BX^{n-1,n}_{\Cf_h} = \tilde\BX^{n-1,n}_{K} \quad \hbox{on}\;\; K\cap\Gamma_\eta^{n-1},\qquad
\tilde\BX^{n-1,n}_{K}\big|_{\Gamma_{K,m}^{n-1}} = \tilde\BX^{n-1,n}_{{K,m}} .
\end{equation} 
\begin{figure}
\centering
\begin{subfigure}[A cut element with markers]{
		\centering
		\begin{tikzpicture}[scale =1.8]
			\draw[thick,fill = white] (0,0) rectangle (2,2); 
			\draw [thick,red,fill=yellow] (0.5,0) to [out=90,in=190](2,1.5) --(2,0)--(0.5,0);
			\draw[thick,black] (2,1.5) --(2,0)--(0.5,0);
			\draw [blue,fill=blue] (0.5,0) circle [radius=0.04];
			\draw [blue,fill=blue] (2,1.5) circle [radius=0.04];	
			\draw[black,fill=black](1,1.05)circle[radius=0.03];
			\draw[black,fill=black](1.2,1.2)circle[radius=0.03];
			\draw[blue,fill=blue] (0.8,0.85) circle [radius=0.04];
			\draw[blue,fill=blue] (1.4,1.3) circle [radius=0.04];
			\node at (0.63,0.9) {\tiny{$\BA_0$}};
			\node at (0.85,1.12) {\tiny{$\BA_1$}};
			\node at (1.,1.3) {\tiny{$\BA_2$}};
			\node at (1.25,1.45) {\tiny{$\BA_3$}};
			\node at (0.9,0.45) {\tiny{$\Gamma_{K,0}^{n-1}$}};  
			\node at (1.25,0.95) {\tiny{$\Gamma_{K,1}^{n-1}$}};
			\node at (1.75,1.25) {\tiny{$\Gamma_{K,2}^{n-1}$}};
		\end{tikzpicture}
	}
\end{subfigure}
\begin{subfigure}[An illustration of $\tilde \BX_{\Cf_h}^{n-1,n}$ for $k=3$.]
	{
		\centering
		\begin{tikzpicture}[scale =1]
			\draw [thick,red] (-1.5,0+0.5) to [out=90,in=190](0,1.6+0.5);
			\draw [thick] (5.5-2.5,0.0+0.3) .. controls (6.6-2.5,0.9+0.3) and (5.5-2.5,1.0+0.3) .. (6.5-2.5,2.4+0.3);
			\draw [thick,blue] (5.45-2.5,0+0.3) .. controls (6.7-2.5,0.9+0.3) and (5.30-2.5,1.0+0.3) .. (6.55-2.5,2.4+0.3);		
			\node at (-0.7,2.7) {\small{$\Gamma^{n-1}_{K,1}$}};

			\draw[blue,fill=blue] (-0.45,1.95) circle [radius=0.06];
			\draw[blue,fill=blue] (-1.48,0.73) circle [radius=0.06];
			\draw[black,fill=black] (-1,1.61) circle [radius=0.05];
			\draw[black,fill=black] (-1.35,1.15) circle [radius=0.05];

			\draw[ultra thick, ->] (-0.4,1.1)--(0.8,1.1);
			\node at (0.25,1.5) {\footnotesize{$\BX_{\Cf_h}^{n-1,n}$}};
			
			\node at (-1.48-0.7,0.73) {\footnotesize{$\BA_0$}};
			\node at (-1.35-0.7,1.15) {\footnotesize{$\BA_1$}};
			\node at (-1-0.7,1.65) {\footnotesize{$\BA_2$}};
			\node at (-0.45-0.5,2.1) {\footnotesize{$\BA_3$}};
			
			\node at (2.12-0.7, 0.3) {\footnotesize{$\BA_0^n$}};
			\node at (2.0-0.7,1.15) {\footnotesize{$\BA_1^{n}$}};
			\node at (2.05-0.6,1.75) {\footnotesize{$\BA_2^n$}};
			\node at (1.7,2.40) {\footnotesize{$\BA_3^n$}};
			
			\draw[fill=green] (1.76,0.32+0.3) circle [radius=0.06]; 
			\draw[fill=green] (1.8,1.08+0.3) circle [radius=0.06]; 
			\draw[fill=green] (2.03,1.68+0.3) circle [radius=0.06]; 
			\draw[fill=green] (2.3,2.2+0.3) circle [radius=0.06];  

			\draw[ultra thick, -latex] (2.2,1.1)--(3.34,1.1);

			\draw[fill=green] (5.72-2.5,0.22+0.3) circle [radius=0.06];
			\draw[fill=green] (6.04-2.5,0.82+0.3) circle [radius=0.06];
			\draw[fill=green] (6.06-2.5,1.45+0.3) circle [radius=0.06];
			\draw[fill=green] (6.3-2.5,2.1+0.3) circle [radius=0.06];
			\node at (4-0.5,2.75) {\small{$\tilde{\Gamma}^{n}_{K,1}$}};
			\node at (4.5-0.5,0.5)[blue] {\small{${\Gamma}^{n}_{\eta}$}};
			
		\end{tikzpicture}
	} 
\end{subfigure}
\caption{Left: A cut element with  markers $\Cp^{n-1}$(blue points), 
	$3$ segments of $\Gamma_{\eta,K}^{n-1}$, (red lines),
	quasi-uniformly distributed points $\BA_i$ on $\Gamma_{K,1}^{n-1}$.
	Right: The black line $\tilde \Gamma_{K,m}^n$  is obtained by Lagrange interpolation based on the points $\BA_i^n$, $0\leq i\leq 3$ (green points), which is approximation of $\Gamma_{\eta}^n$ (blue line).
}
\label{fig:cut element and illustration}
\end{figure}
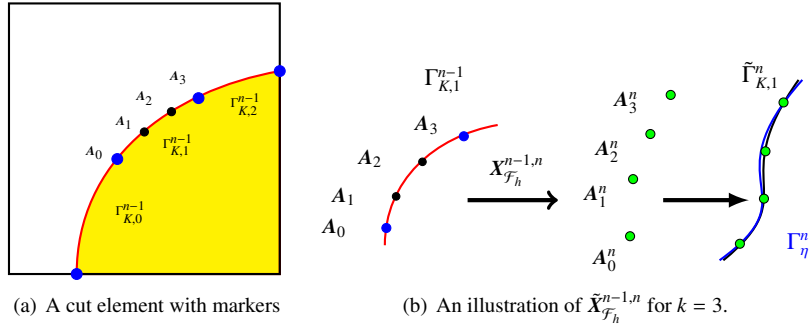 

{\bf Step~2. Construct the backward flow map $\BX^{n,n-1}_{\Cf_h}$.}
It is easy to see that $\tilde \BX_{K}^{n-1,n}$ is the $k^{\text{th}}$-order Lagrange interpolation of $\BX_{\Cf_h}^{n-1,n}(\Gamma_{\eta,K}^{n-1})$.
It is natural to use the inverse of $\tilde \BX_{\Cf_h}^{n-1,n}$ to approximate 
$(\BX_{\Cf_h}^{n-1,n})^{-1}$.
Finally, we define the backward flow map by
$\BX_{\Cf_h}^{n,n-1} := (\tilde \BX_{\Cf_h}^{n-1,n})^{-1}\circ \BP_n$, 
where $\BP_n:\Gamma_{\eta}^n \to \tilde\Gamma_{\eta}^{n}$ is the projection defined by $\BP_n(\Bx) := \argmin\limits_{\By\in \tilde\Gamma_{\eta}^n} \SN{\By-\Bx}$ for any $\Bx\in \Gamma_\eta^n$.

The construction of $\BX_{\Cf_h}^{n,n-1}$ is similar to \cite[Section~3.4]{maz21} but much more computationaly economic, since we only need the boundary map from $\Gamma_{\eta}^n$ to $\Gamma_{\eta}^{n-1}$ instead of the volume map from $\Omega_{\eta}^n$ to $\Omega_{\eta}^{n-1}$. Finally, the ALE map $\BX_h^{n,n-1}\in \BV(k,\Ct_h^n)$ is defined as in \eqref{eq:discrte Xh} by letting $\Bg_{\eta}^{n,n-1}=\BX_{\Cf_h}^{n,n-1}$, and $\Ca_{k,h}^n$, $\Bw_{k,h}^n$, and $\BX_h^{n,n-i}$ are obtained similarly as in \eqref{eq: ALE Ah}--\eqref{def:Xh n-i} for $i=2,\cdots,k$.
\subsection{The discrete scheme}
In view of \eqref{eq:semi equa}, the discrete scheme for solving \eqref{eq:nonlienar} is to
find $\Bv_h^n \in \BV(k,\Ct_h^n)$ such that 
\begin{equation*} 
\frac{1}{\tau}(\Lambdabf_{\Bw_h}^k \ul{\BV_h^n} ,\varphibf_h)_{\Omega_\eta^n} 
+( \nabla \Bv_h^n,\nabla \varphibf_h)_{\Omega_\eta^n} 
+\mathscr{J}_1^n(\Bv_h^n,\varphibf_h)
=(\Bf^n, \varphibf_h)_{\Omega_\eta^n}+\left<\Bg_{N}, \varphibf_h\right>_{\Gamma_{\eta}^n},
\end{equation*}
for any $\varphibf_h\in \BV(k,\Ct_h^n)$, where
\begin{equation*}
\frac{1}{\tau} \Lambdabf_{\Bw_h}^k \ul{\BV_h^n} 
= \frac{1}{\tau}\sum_{i=0}^k a_i^k \Bv_h^{n-i}\circ \BX_h^{n,n-i} - \Bw_{k,h}^n \cdot \nabla \Bv_h^n.
\end{equation*} 

\section{A two-phase problem}
In this section, we focus on a two-phase model with a time-varying interface. 
Let $D\subset\bbR^2$ be an open square with boundary $\Sigma=\partial D$.
For any $t\geq 0$, let $\Omega_{1,t}$ and $\Omega_{2,t}$ be two time-varying sub-domains of $D$ occupied by two immiscible fluids, respectively.
We assume $\Gamma_t = \partial \Omega_{1,t}\subset D$ and
$\partial \Omega_{2,t}=\Sigma \cup \Gamma_t$.
Consider the linear interface problem:
\begin{equation}\label{eq:interface}
\frac{\partial u_i}{\partial t}  -\nu_i \Delta u_i = f_i\;\;\; \text{in}\; \Omega_{i,t},\;\; i=1,2,\quad 
\jump{u}=g_D\;\; \text{on}\;\Gamma_{t},\quad
\jump{\partial_{\Bn} u}=g_N \;\; \text{on}\;\Gamma_{t},
\end{equation}
where $\jump{u}=u_1-u_2$ denotes the jump of $u$ across $\Gamma_t$, 
the viscosities $\nu_i$ are positive constants,  
and $\Bn$ is the unit normal on $\Gamma_t$ pointing to $\Omega_{2,t}$.
The interface $\Gamma_t$ is driven by $\Bv$ 
and has the same form as \eqref{eq:Gammat}.
We treat the model as two free-boundary problems which are coupled with the interface conditions. 

\subsection{Finite element spaces}
Suppose we have obtained the approximate interface $\Gamma_\eta^n$ 
by using some interface-tracking algorithm mentioned previously.
Let $\Omega_{\eta,1}^n$ be an approximate domain of 
$\Omega_1^n=\Omega_{1,t_n}$ such that
$\Gamma_\eta^n = \partial \Omega_{\eta,1}^n$.
Define $\Omega_{\eta,2}^n = D\backslash \Omega_{\eta,1}^n$.
Similar to the single phase case, we define
\begin{equation} \label{eq:interface delta omega}
\Omega_{\delta,i}^n = :\Big\{\Bx \in \bbR^2: \min_{\By\in \Omega_{\eta,i}^n}|\Bx -\By| \leq  0.5 \tau\Big\}, \quad i=1,2.
\end{equation}
Let $\Ct_h$ be the uniform partition of $D$ into closed squares of side-length $h$. It generates the covers of $\Omega^n_{\delta,i}$, $i=1,2$ and the cover of $\Gamma^n_{\eta}$
\begin{align*}
\Ct^n_{h,i} := \left\{K\in \Ct_h^n: \bar{K} \cap \bar{\Omega}_{\delta,i}^n \neq \emptyset\right\},
\quad
\Ct^n_{B,i} := \left\{K\in \Ct_h:\; 
\bar K\cap(\Gamma^n_{\eta}\cup\partial \Omega_{\delta,i}^n) 
\neq \emptyset \right\}.
\end{align*}
We define $\tilde{\Omega}_{i}^n:=\bigcup_{K\in\Ct^n_{h,i}}K$.
Clearly $\Omega^n_{\eta,i}\subset\tilde{\Omega}^n_{i}$. Define
\begin{equation*}
\Ce_{i,B}^{n}= \big\{E\in\Ce_h: \; E \not\subset \partial\tilde\Omega^n_{i}
\;\; \hbox{and}\;\; \exists K\in \Ct^n_{B,i}\;\;
\hbox{s.t.}\;\; E\subset\partial K\big\}.
\end{equation*}
The mesh is shown in Fig.~\ref{fig:interface Omega mesh}. We define the finite element spaces
\begin{align*}
&\Cw_{h,0} = \{v\in H^1_0(D), v|_K\in Q_k(K),\forall K\in \Ct_h\},\\
&\mathcal{\BW}_h^n:=  \big\{(v_{h,1},v_{h,2}):
v_{h,i} =v_h|_{\tilde\Omega^n_i},\; v_h\in \Cw_{h,0},\; i=1,2\big\}. 
\end{align*}

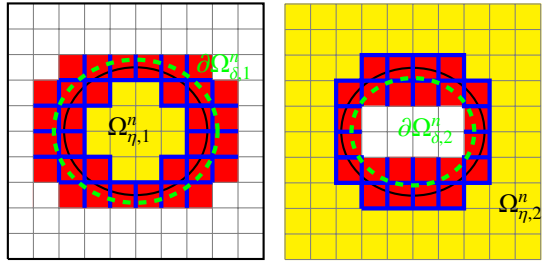
\begin{figure}[t]
\centering
\begin{subfigure}
	{		
		\centering
		\begin{tikzpicture}[scale =1.7] \filldraw[red](0.2*3,0.2*2)--(0.2*3,0.2*3)--(0.2*2,0.2*3)--(0.2*2,0.2*7)--(0.2*3,0.2*7)--(0.2*3,0.2*8)--(0.2*4,0.2*8)--(0.2*7,0.2*8)--(0.2*7,0.2*7)--(0.2*8,0.2*7)--(0.2*8,0.2*3)--(0.2*7,0.2*3)--(0.2*7,0.2*2)--(0.2*3,0.2*2);  		\filldraw[yellow](0.2*3,0.2*6)--(0.2*4,0.2*6)--(0.2*4,0.2*7)--(0.2*5,0.2*7)--(0.2*6,0.2*7)--(0.2*6,0.2*6)--(0.2*7,0.2*6)--(0.2*7,0.2*5)--(0.2*7,0.2*4)--(0.2*6,0.2*4)--(0.2*6,0.2*3)--(0.2*5,0.2*3)--(0.2*4,0.2*3)--(0.2*4,0.2*4)--(0.2*3,0.2*4)--(0.2*3,0.2*5)--(0.2*3,0.2*6);
			\filldraw[red](0.2*1,0.2*3)--(0.2*2,0.2*3)--(0.2*2,0.2*7)--(0.2*1,0.2*7)--(0.2*1,0.2*2);
			\filldraw[red](0.2*8,0.2*3)--(0.2*9,0.2*3)--(0.2*9,0.2*7)--(0.2*8,0.2*7)--(0.2*8,0.2*3);
			\filldraw[red](0.2*2,0.2*7)--(0.2*3,0.2*7)--(0.2*3,0.2*8)--(0.2*2,0.2*8)--(0.2*2,0.2*7);
			\filldraw[red](0.2*7,0.2*7)--(0.2*8,0.2*7)--(0.2*8,0.2*8)--(0.2*7,0.2*8)--(0.2*7,0.2*7);
			\filldraw[red](0.2*2,0.2*2)--(0.2*3,0.2*2)--(0.2*3,0.2*3)--(0.2*2,0.2*3)--(0.2*2,0.2*2);
			\filldraw[red](0.2*7,0.2*2)--(0.2*8,0.2*2)--(0.2*8,0.2*3)--(0.2*7,0.2*3)--(0.2*7,0.2*2);
			\draw[black, thick] (1,1) ellipse [x radius=0.56cm, y radius=0.5cm];
			\filldraw[step =0.2cm,gray,thin] (0,0) grid (2cm,2cm);
			\draw [blue, ultra thick] (0.2*3,0.2*3)--(0.2*3,0.2*7)--(0.2*7,0.2*7)--(0.2*7,0.2*3)--(0.2*3, 0.2*3);
			\draw [blue,ultra thick] (0.2*1,0.2*4)--(0.2*3,0.2*4);
			\draw [blue, ultra thick] (0.2*1,0.2*5)--(0.2*3,0.2*5);
			\draw [blue,ultra thick] (0.2*1,0.2*6)--(0.2*3,0.2*6);
			\draw [blue,ultra thick] (0.2*2,0.2*4)--(0.2*3,0.2*4);
			\draw [blue,ultra thick] (0.2*2,0.2*3)--(0.2*2,0.2*7);
			\draw [blue,ultra thick] (0.2*4,0.2*2)--(0.2*4,0.2*4)--(0.2*3, 0.2*4);
			\draw [blue,ultra thick] (0.2*4,0.2*8)--(0.2*4,0.2*6)--(0.2*3, 0.2*6);
			\draw [blue,ultra thick] (0.2*6,0.2*2)--(0.2*6,0.2*4)--(0.2*8, 0.2*4);
			\draw [blue,ultra thick] (0.2*6,0.2*8)--(0.2*6,0.2*6)--(0.2*8, 0.2*6);
			\draw [blue,ultra thick] (0.2*5,0.2*2)--(0.2*5,0.2*3);
			\draw [blue,ultra thick] (0.2*7,0.2*5)--(0.2*9,0.2*5);
			\draw [blue,ultra thick] (0.2*5,0.2*7)--(0.2*5,0.2*8);
			\draw [blue,ultra thick] (0.2*7,0.2*4)--(0.2*9,0.2*4);
			\draw [blue,ultra thick] (0.2*7,0.2*6)--(0.2*9,0.2*6);
			\draw [blue,ultra thick] (0.2*2,0.2*3)--(0.2*3,0.2*3);
			\draw [blue,ultra thick] (0.2*2,0.2*7)--(0.2*3,0.2*7);
			\draw [blue,ultra thick] (0.2*7,0.2*7)--(0.2*8,0.2*7);
			\draw [blue,ultra thick] (0.2*7,0.2*3)--(0.2*8,0.2*3);
			\draw [blue,ultra thick] (0.2*8,0.2*3)--(0.2*8,0.2*7);
			\draw [blue,ultra thick] (0.2*3,0.2*2)--(0.2*3,0.2*3);
			\draw [blue,ultra thick] (0.2*3,0.2*7)--(0.2*3,0.2*8);
			\draw [blue,ultra thick] (0.2*7,0.2*2)--(0.2*7,0.2*3);
			\draw [blue,ultra thick] (0.2*7,0.2*7)--(0.2*7,0.2*8);
			\node[left] at (0.2*6,0.2*5) {$\Omega_{\eta,1}^n$};
			\node[right,green,thick] at (0.2*7,0.2*7.5) {$\partial \Omega_{\delta,1}^n$};
			\draw[ green, ultra thick,dashed] (1,1) ellipse [x radius=0.64cm, y radius=0.56cm];
			\draw[black,thick](0,0)--(2,0)--(2,2)--(0,2)--(0,0);
		\end{tikzpicture}\label{fig:sub1}
	}
\end{subfigure}%
\begin{subfigure}
	{		\centering
		\begin{tikzpicture}[scale =1.7]
			\filldraw[yellow](0,0)--(0.2*10,0)--(0.2*10,0.2*10)--(0,0.2*10)--(0,0);
			\filldraw[red](0.2*3,0.2*2)--(0.2*3,0.2*3)--(0.2*2,0.2*3)--(0.2*2,0.2*7)--(0.2*3,0.2*7)--(0.2*3,0.2*8)--(0.2*4,0.2*8)--(0.2*7,0.2*8)--(0.2*7,0.2*7)--(0.2*8,0.2*7)--(0.2*8,0.2*3)--(0.2*7,0.2*3)--(0.2*7,0.2*2)--(0.2*3,0.2*2);  		
			\filldraw[white](0.2*3,0.2*6)--(0.2*4,0.2*6)--(0.2*4,0.2*7)--(0.2*5,0.2*7)--(0.2*6,0.2*7)--(0.2*6,0.2*6)--(0.2*7,0.2*6)--(0.2*7,0.2*5)--(0.2*7,0.2*4)--(0.2*6,0.2*4)--(0.2*6,0.2*3)--(0.2*5,0.2*3)--(0.2*4,0.2*3)--(0.2*4,0.2*4)--(0.2*3,0.2*4)--(0.2*3,0.2*5)--(0.2*3,0.2*6);
			\filldraw[red](0.2*4,0.2*3)--(0.2*6,0.2*3)--(0.2*6,0.2*4)--(0.2*4,0.2*4)--(0.2*4,0.2*3);
			\filldraw[red](0.2*4,0.2*6)--(0.2*6,0.2*6)--(0.2*6,0.2*7)--(0.2*4,0.2*7)--(0.2*4,0.2*6);
			\draw[black, thick] (1,1) ellipse [x radius=0.56cm, y radius=0.5cm];
			\filldraw[step =0.2cm,gray,thin] (0,0) grid (2cm,2cm);
			\draw[blue,ultra thick](0.2*3,0.2*2)--(0.2*7,0.2*2)--(0.2*7,0.2*3)--(0.2*8,0.2*3)--(0.2*8,0.2*7)--(0.2*7,0.2*7)--(0.2*7,0.2*8)--(0.2*3,0.2*8)--(0.2*3,0.2*7)--(0.2*2,0.2*7)--(0.2*2,0.2*3)--(0.2*3,0.2*3)--(0.2*3,0.2*2);
			\draw [blue,ultra thick] (0.2*4,0.2*8)--(0.2*4,0.2*7)--(0.2*3,0.2*7)--(0.2*3,0.2*6);
			\draw [blue,ultra thick] (0.2*2,0.2*4)--(0.2*3,0.2*4);
			\draw [blue, ultra thick] (0.2*2,0.2*5)--(0.2*3,0.2*5);
			\draw [blue,ultra thick] (0.2*2,0.2*6)--(0.2*3,0.2*6);
			\draw [blue,ultra thick] (0.2*2,0.2*4)--(0.2*3,0.2*4)--(0.2*3,0.2*3)--(0.2*4,0.2*3);
			\draw [blue,ultra thick] (0.2*6,0.2*7)--(0.2*7,0.2*7)--(0.2*7, 0.2*6)--(0.2*8,0.2*6);
			\draw [blue,ultra thick] (0.2*6,0.2*3)--(0.2*7,0.2*3)--(0.2*7, 0.2*4)--(0.2*8,0.2*4);
			\draw [blue,ultra thick] (0.2*6,0.2*8)--(0.2*6,0.2*7);
			\draw [blue,ultra thick] (0.2*5,0.2*8)--(0.2*5,0.2*7);
			\draw [blue,ultra thick] (0.2*4,0.2*2)--(0.2*4,0.2*3);
			\draw [blue,ultra thick] (0.2*3,0.2*2)--(0.2*3,0.2*3);
			\draw [blue,ultra thick] (0.2*5,0.2*2)--(0.2*5,0.2*3);
			\draw [blue,ultra thick] (0.2*6,0.2*2)--(0.2*6,0.2*3);
			\draw [blue,ultra thick] (0.2*7,0.2*5)--(0.2*8,0.2*5);
			\draw [blue,ultra thick] (0.2*5,0.2*7)--(0.2*5,0.2*8);	
			\draw [blue,ultra thick] (0.2*4,0.2*3)--(0.2*4,0.2*4);	
			\draw [blue,ultra thick] (0.2*5,0.2*3)--(0.2*5,0.2*4);	
			\draw [blue,ultra thick] (0.2*6,0.2*3)--(0.2*6,0.2*4);	
			\draw [blue,ultra thick] (0.2*4,0.2*6)--(0.2*4,0.2*7);	
			\draw [blue,ultra thick] (0.2*5,0.2*6)--(0.2*5,0.2*7);	
			\draw [blue,ultra thick] (0.2*4,0.2*3)--(0.2*6,0.2*3);	
			\draw [blue,ultra thick] (0.2*4,0.2*7)--(0.2*6,0.2*7);	
			\draw [blue,ultra thick] (0.2*6,0.2*6)--(0.2*6,0.2*7);	
			\node[right] at (0.2*8,0.2*2) {$\Omega_{\eta,2}^n$};
			\node[right,green,ultra thick] at (0.2*4,0.2*5) {$\partial\Omega_{\delta,2}^n$};
			\draw[ green,ultra thick,dashed] (1,1) ellipse [x radius=0.48cm, y radius=0.42cm];
		\end{tikzpicture}\label{fig:sub2}
	}
\end{subfigure}
\caption{ Left figure: $\tilde{\Omega}_1^n$ (red and yellow squares), $\Ct_{B,1}^n$ (the set of red squares), $\Ce_{1,B}^n$ (blue edges) and $\partial \Omega_{\delta,1}^n$ (green edge).
	Right figure: $\tilde{\Omega}_2^n$ (red and yellow squares), $\Ct_{B,2}^n$ (the set of red squares), $\Ce_{2,B}^n$ (blue edges) and $\partial \Omega_{\delta,2}^n$ (green edge).	
} \label{fig:interface Omega mesh}
\end{figure}

\subsection{Discrete ALE map and fully discrete scheme}

The ALE-UFE framework is used to each phase of the interface problem, and leads to the fully discrete scheme by coupling the discrete formulations of both phases with interface conditions.

First we construct ALE maps piecewise.
The discrete ALE map $\Ca_{k,h,1}^n$, $\Bw_{k,h,1}^n$ are exactly the same as $\Ca_{k,h}^n$ and $\Bw_{k,h}^n$ defined in 
section~\ref{sec:construction ALE}. To construct $\Ca_{k,h,2}^n$, we let
$\BX_{h,2}^{n,n-1}\in \BV(k,\Ct_{h,2}^n)$ be the solution satisfying $\BX_{h,2}^{n,n-1}\big|_{\partial D} = \BI$ and 
\begin{equation}\label{eq:Xh1}
\mathscr{A}_h^n(\BX_{h,2}^{n,n-1},\Bv_h) =
\Cf_{\Gamma_\eta^n}(\Bg_\eta^{n,n-1},\Bv_h)
\qquad \forall \Bv_h \in \BV_0(k,\Ct_{h,2}^n),
\end{equation}
where $\BV_0(k,\Ct_{h,2}^n):=\{\Bv\in \BV(k,\Ct_{h,2}^n);\Bv|_{\partial D}=0 \}$ and $\Bg_{\eta}^{n,n-1}$ is defined in section~\ref{sec:construction ALE}.
Then we define $\Ca_{k,h,2}^n = \sum_{i=0}^k l_n^i \BX_{h,2}^{n,n-i}$ and $\Bw_{k,h,2}^n = \sum_{i=0}^k (l_n^i)' \BX_{h,2}^{n,n-i}$.

Next we use the backward flow maps   
to discretize the time derivatives 
\begin{equation*}
\frac{1}{\tau}\Lambda_{\Bw_{h,j}}^k \ul{\BU^n_{h,j}}
= \frac{1}{\tau} \Lambda^k\ul{\BU^n_{h,j}} 
-\Bw_{k,h,j}^n\cdot \nabla u_{h,j}^n,\quad 
\Lambda^k\ul{\BU^n_{h,j}} := \sum_{i=0}^k \lambda_i^k u_{h,j}^{n-i}\circ \BX_{h,j}^{n,n-i}.
\end{equation*}
The discrete scheme is to find $u_h^n \in \BW_h^n$ such that
\begin{align*}
\sum_{j=1}^2\Big[
\Big(\frac{1}{\tau}\Lambda^k_{\Bw_{h,j}}\ul{\BU^n_{h,j}} -f^n_h,v_{j}\Big)_{\Omega_{\eta,j}^n}
+\mathscr{A}_{j}^n(u_{h,j}^n,v_j)\Big]
+\mathscr{S}_I^n(u_h^n,v)
+\mathscr{J}_{I}^n(u_h^n,v)  = 0,
\end{align*}
for all $v \in \BW_h^n$ with $v_j=v|_{\Omega_{\eta,j}^n}$, where 
\begin{align*}
\mathscr{A}_{j}^n(u,v) &:= 
(\nu_j \nabla u,\,\nabla v)_{\Omega_{\eta,j}^n}
+\mathscr{J}_{1,j}(u,v),\\
\mathscr{J}_{1,j}(u,v)&:=
\sum_{E\in \Ce_{j,B}^{n}} 
\sum_{l=1}^k h^{2l-1}\left< \nu_j 
\jump{\partial_{\Bn}^l u_{j}},
\jump{\partial_{\Bn}^l v_{j}}\right>_E.
\end{align*}
Moreover, $\mathscr{S}_{I}^n$ and $\mathscr{J}_{I}^n$ are bilinear forms which combine the two phases
\begin{align*}
\mathscr{S}_{I}^n(u,v) :=\,& 
-\left<  \avg{\nu\partial_{\Bn} u},\jump{v} \right>_{{\Gamma_\eta^n}}
-\left<\avg{\nu \partial_{\Bn} v},\jump{u}\right>_{{\Gamma_\eta^n}}, \\ 
\mathscr{J}_{I}^n(u,v) :=\,& \gamma_0 h^{-1}\avg{\nu} 
\left<\jump{u},\jump{v}\right>_{\Gamma_\eta^n}. 
\end{align*}
where $\gamma_0$ is a positive penalty coefficient and
$\Bn$ is the unit outward normal to $\Gamma_\eta^n$ from $\Omega_{\eta,1}^n$ to $\Omega_{\eta,2}^n$.
Here we have used the average operator
\begin{equation*}
\avg{a} = \kappa_1 a_1 + \kappa_2 a_2,\quad 
\kappa_1 = \nu_2/({\nu_1+\nu_2}),\quad \kappa_2 ={\nu_1}/({\nu_1+\nu_2}).
\end{equation*} 

\section{Numerical experiments}
\label{sec:num}

In this section, we demonstrate the ALE-UFE method with all 
the models in the previous sections. 
Throughout the section, we choose $\gamma_0=1000$. 
To simplify computations, we set the pre-calculated initial values by the exact solution, namely, $u^j_h=u(\cdot,t_j)$, for $0\le j\le k-1$.
\subsection{One-phase linear problem}
\label{subsec:Free boundary problems}
In this section, we consider a one-phase problem where 
the velocity of the moving boundary is given by an analytic function.
The exact solution is set by $u = \sin(\pi (x_1+t))\sin(\pi (x_2+t))$. 
The right-hand side and the boundary values in \eqref{cd-model} are set by $u$. The mesh of the evolution domain is shown in Fig.~\ref{fig:Omega_linear}. The boundary is varying according to 
\begin{equation*}
x_1(t)= \frac{x_1(0)}{1+0.2\sin(2t)}+\frac{\sin(2t)}{16},\quad
x_2(t)= \frac{x_2(0)}{1-0.25\sin(2t)}+\frac{\sin(2t)}{16}.
\end{equation*}
The initial boundary is a circle centered at $(0.5,0.5)^{\top}$ with radius $1/8$. The numerical error is measured by
\begin{equation*}
e^N=\bigg(\|u(\cdot,T)-u_h^{N}\|_{L^2(\Omega_\eta^N)}^2
+\tau \sum_{n=k}^N |u-u_h^n|_{H^1(\Omega_\eta^n)}^2\bigg)^{1/2}.
\end{equation*}
\begin{figure}[htbp]
\centering
\includegraphics[width=1\textwidth]{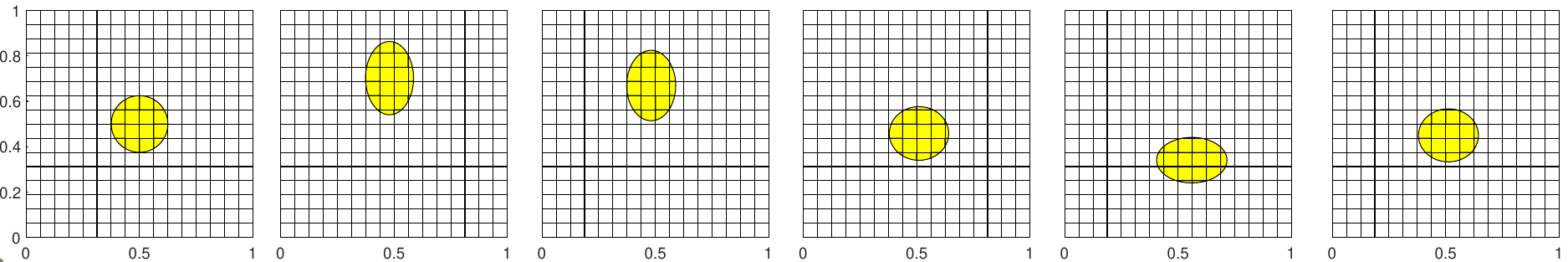}
\caption{\small The moving domain $\Omega_{\eta}^n$ at $t_n=0$, $0.2$, $0.4$, $0.6$, $0.8$ and $1$ ($h=1/16$).}	\label{fig:Omega_linear}
\end{figure}
Numerical results for $k=3,4$ are shown in Tables~\ref{tab:conv}. Optimal convergence rates $e^N =O(\tau^k)$ are obtained for both the third- and fourth-order methods.
Although the location of the boundary is given explicitly, the UCFE method in \cite{ma21} is not applicable in this case due to the lack of flow velocity inside the domain and the absence of  moving domains that maintain the same volume.
Therefore, the ALE-UFE method has a wider range of applications than the UCFE mehtod. 
\begin{table}[htbp]
\center 
\setlength{\tabcolsep}{12.5mm} 
\caption{Convergence rates for section~\ref{subsec:Free boundary problems}.} \label{tab:conv}
\begin{tabular}{ lcccc}
	\toprule[2pt]
	$h=\tau$     &$e^N\; (k=3)$       &rate   &$e^N\; (k=4)$       &rate      \\ \toprule[2pt] 
	1/16 &6.16e-03 &-     &1.91e-3 &-             \\ 
	1/32 &7.94e-04 &2.95  &1.25e-4 &3.93 \\ 
	1/64 &1.00e-04 &2.98  &9.97e-6 &3.97 \\ 
	1/128&1.25e-05 &2.99  &5.01e-7 &3.98\\  \toprule[2pt] 
\end{tabular}
\end{table}

\subsection{Nonlinear problem}\label{nonlinear}

Now we demonstrate the ALE-UFE method with a nonlinear problem where the boundary motion is specified by the solution. We require $\Omega_T=\Omega_0$ to check the accuracy of interface-tracking. 
The exact solution is set by
\begin{equation*}
\Bv = \cos(\pi t/3)\big[\sin(2\pi x_1)\sin(2\pi x_2), \cos(2\pi x_1)\cos(2\pi x_2)\big]^{\top}.
\end{equation*}
The initial domain is a disk with radius $R=0.15$ and centered at $(0.5,0.5)$. 
The domain is stretched into an inverted U shape at $t=1.5$ and returns to its initial shape at $T=3$ (see Fig.~\ref{fig:Omega_nonlinear}).
\begin{figure}[htbp]
\centering
\includegraphics[width=1\textwidth]{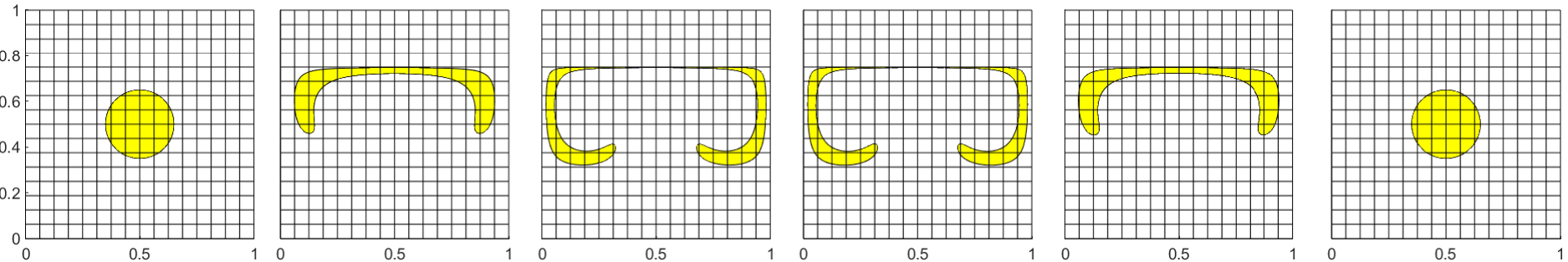}
\caption{\small Approximate domains at $t_n =0$, $0.6$, $1.2$, $1.8$, $2.4$ and $3$ ($h=1/16$).}	\label{fig:Omega_nonlinear}
\end{figure}

The source term and the Neumann condition of \eqref{eq:nonlienar} are given by 
\begin{equation*}
\Bf = \partial_t \Bv -\Delta \Bv \quad \text{in}\; \Omega_t, \quad
\Bg_{N} =\nabla \Bv \cdot \Bn \quad \text{on } \; \Gamma_t.
\end{equation*}
Define the boundary-tracking errors
\begin{align*}
e_{0,\Omega} &= \sum_{K\in \Ct_h} |\text{area}(\Omega_T\cap K)-\text{area}(\Omega_{\eta}^N\cap K)|,\\
e_{1,\Omega} &= \bigg[\sum_{n=k}^N \tau \Big( \sum_{K\in \Ct_h}
\big|\text{area}(\Omega_{\text{ref}}^n\cap K)-\text{area}(\Omega_{\eta}^n\cap K)\big|\Big)^2\bigg]^{1/2}.
\end{align*}
Here $\Omega_{\text{ref}}^n$ is the reference domain computed with the finest grid and the smallest time step.
Approximation errors are computed on the numerically tracked domain
\begin{equation*} 
e_0 = \bigg(\sum_{n=k}^N \tau \|\Bv(\cdot,t_n)-\Bv_h^n\|^2_{\BL^2(\Omega_{\text{ref}}^n)}\bigg)^{1/2},\quad 
e_1 = \bigg(\sum_{n=k}^N \tau |\Bv(\cdot,t_n)-\Bv_h^n|^2_{\BH^1(\Omega_{\text{ref}}^n)}\bigg)^{1/2}.
\end{equation*}
Tables~\ref{tab:nonlinear k3} and ~\ref{tab:nonlinear k4} show that optimal convergence rates are obtained for the numerical solutions and the boundary-tracking algorithm as $h$ (or $\tau$) approaches $0$.

\begin{table}[t]
\setlength{\tabcolsep}{5mm} 
\caption{Convergence rates for section~\ref{nonlinear} ($k=3$).} \label{tab:nonlinear k3} 
\begin{tabular}{ lcccccccc}
	\toprule[2pt]
	$h=\tau$     &$e_0$       &rate   &$e_1$       &rate  &$e_{0,\Omega}$    &rate  &$e_{1,\Omega}$    &rate \\ \toprule[2pt] 
	1/32&1.11e-04 &- &4.06e-04 &-    &1.21e-03 &-     &1.40e-03  &-\\
	1/64&1.20e-05 &3.21 &4.42e-05 &3.19  &1.65e-04 &2.88 &1.71e-04 &3.03\\
	1/128&1.41e-06 &3.08 &5.50e-06 &3.00 &2.13e-05 &2.95 &2.11e-05 &3.01\\
	1/256&1.72e-07 &3.03 &6.88e-07 &3.00  &2.67e-06 &2.99&2.63e-06 &3.00\\ \toprule[2pt]
\end{tabular}
\end{table} 
\begin{table}[t]
\setlength{\tabcolsep}{5mm} 
\caption{Convergence rates for section~\ref{nonlinear} ($k=4$)} \label{tab:nonlinear k4} 
\begin{tabular}{ lcccccccc}
	\toprule[2pt]
	$h=\tau$     &$e_0$       &rate   &$e_1$       &rate  &$e_{0,\Omega}$    &rate  &$e_{1,\Omega}$    &rate \\ \toprule[2pt]
	1/32&4.72e-05 &- &2.34e-04 &-    &4.05e-04 &-     &4.48e-04  &-\\
	1/64&4.14e-06 &3.51 &1.47e-05 &3.99  &3.93e-05 &3.36 &3.84e-05 &3.54\\
	1/128&2.82e-07 &3.87 &4.39e-07 &5.06 &2.73e-06 &3.85 &2.78e-06 &3.78\\
	1/256&1.80e-08 &3.96 &2.76e-08 &3.99  &1.75e-07 &3.96&1.86e-07 &3.90\\ \toprule[2pt]
\end{tabular}
\end{table}

\subsection{Two-phase flow}
\label{ex:linear2}
\begin{figure}[htbp]
	\centering
	\includegraphics[width=1\textwidth]{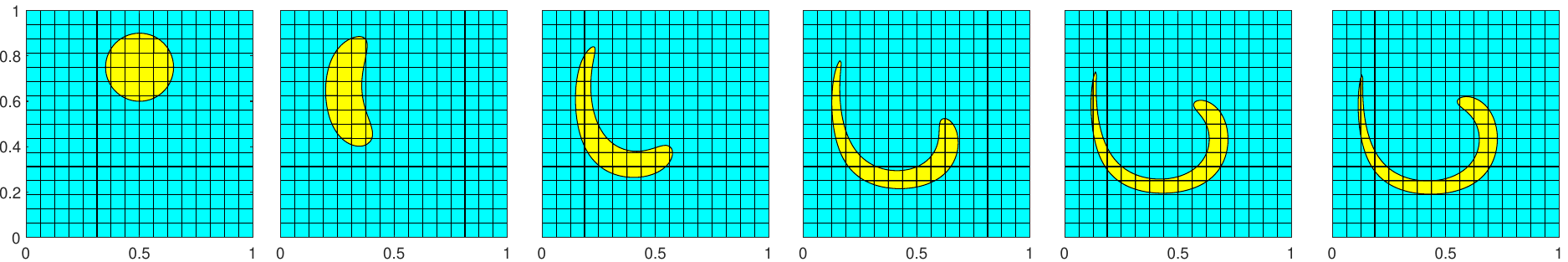}
	\caption{\small Two moving subdomains $\Omega_{\eta,1}^n$, $\Omega_{\eta,2}^n$ at $t_n =0$, $0.3$, $0.6$, $0.9$, $1.2$ and $1.5$ ($h=1/16$).}	\label{fig:Omega_interface}
\end{figure}
We use the cubic MARS algorithm in \cite{zha18} to track the interface and construct $\Omega_{\eta,1}^n$
and $\Omega_{\eta,2}^n = D\backslash \Omega_{\eta,1}^n$.
The initial domain $\Omega_{1,0}$ is a disk of radius $R = 0.15$ at $(0.5,0.75)$. The flow velocity is set by
\begin{equation*} 
\Bv = \cos(\pi t/3)\big[\sin^2 (\pi x_1 )\sin(2\pi x_2),-\sin^2(\pi x_2)\sin(2\pi x_1)\big]^{\top}
\end{equation*}
The domain is stretched into a snake-like region at $T=1.5$ (see Fig.~\ref{fig:Omega_interface}). 
The viscosities are set by $\nu_1 =1000$ and $\nu_2 = 1$.
The exact solution is set by
\begin{equation*}
u_1 = \sin(\pi(x_1+t))\sin(\pi(x_2+t)),\quad
u_2 = e^x\sin(\pi (x_2+t)).
\end{equation*}
The numerical error is measured by $e^N$, where
\begin{equation*}
(e^N)^2= \sum_{i=1}^2\bigg[\|u(T)-u_h^{N}\|_{L^2(\Omega_{\eta,i}^N)}^2
+\tau \sum_{n=k}^N \nu_i|u(t_n)-u_h^n|_{H^1(\Omega_{\eta,i}^n)}^2\bigg].
\end{equation*}
Convergence orders for $k=3$ and $4$ are shown in Table~\ref{tab:con interface}. For such a large 
deformation of the domain, the method still yields optimal convergence rates.
\begin{table}[htbp]
\setlength{\tabcolsep}{12.5mm} \center
\caption{ Convergence rates for section~\ref{ex:linear2}.} \label{tab:con interface}
\begin{tabular}{ lcccc}
	\toprule[2pt]
	$h=\tau$     &$e^N\; (k=3)$       &rate   &$e^N\; (k=4)$       &rate      \\ \toprule[2pt] 
	1/16 &6.05e-03 &-     &7.87e-3 &-             \\ 
	1/32 &8.78e-04 &2.78  &5.17e-5 &3.92 \\ 
	1/64 &1.15e-04 &2.92  &3.27e-6 &3.97 \\ 
	1/128&1.47e-05 &2.97  &2.07e-7 &3.97\\ \toprule[2pt]
\end{tabular}
\end{table} 

\subsection{Domain with topological change}
\label{sec:Problems with a topological change}
Finally we consider the domain having a topological change. Its boundary is given by the level set of the function
\begin{align*}
\phi(\Bx,t) = \min\{|\Bx-\Bc_1(t)|,|\Bx-\Bc_2(t)|\} - 0.15,
\end{align*}
where $\Bc_1(t)=[0.5,0.75-0.5t]^{\top}$ and $\Bc_2(t) =[0.5,0.25+0.5t]^{\top}$ are the centers of the two circles.
The finial time is set by $T=1$ which satisfies 
$\phi(\Bx,0) =\phi(\Bx,T)$. The evolution of the domain is shown in Fig.~\ref{fig:Omega_linear two circle}.
\begin{figure}[htbp]
\centering
\includegraphics[width=1\textwidth]{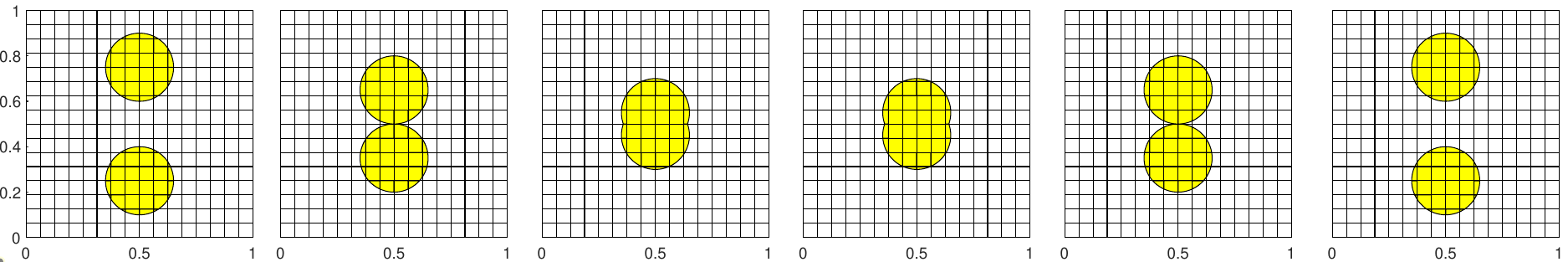}
\caption{\small The moving domain $\Omega_{\eta}^n$ at $t_n=0$, $0.2$, $0.4$, $0.6$, $0.8$, and $1$ ($h=1/16$).}	\label{fig:Omega_linear two circle}
\end{figure}

\begin{table}[htbp]
\center 
\setlength{\tabcolsep}{13mm} 
\caption{Convergence rates for section~\ref{sec:Problems with a topological change} ($k=3$).} \label{tab:topolical change} 
\begin{tabular}{ lcccc}
	\toprule[2pt]
	$h=\tau$     &$e_0$       &rate   &$e_1$       &rate      \\ \toprule[2pt] 
	1/16 &4.95e-05 &-     &6.37e-4 &-             \\ 
	1/32 &1.21e-05 &2.03  &1.52e-4 &2.02 \\ 
	1/64 &3.06e-06 &1.98  &4.36e-5 &1.85 \\ 
	1/128&7.80e-07 &1.98  &1.09e-5 &2.00 \\\toprule[2pt]
\end{tabular}
\end{table} 

The exact solution is $u = \sin(\pi (x_1+t))\sin(\pi (x_2+t))$.
In the implementation of the ALE map, we choose $\Bg_{\eta}^{n,n-1}$ by 
the closet point mapping. 
Again we compute the $L^2$- and $H^1$-errors of the solution on the approximate domain $\Omega_{\eta}^n$. 
Since the evolution of the domain is discontinuous, the assumption in \eqref{assum ALE mapping} does not hold anymore. 
Table ~\ref{tab:topolical change} shows that the convergence rate deteriorates into second order.

\section{Conclusions}
An arbitrary Lagrangian-Eulerian unfitted finite element (ALE-UFE) method has been presented for solving PDEs on time-varying domains. High-order convergence is obtained by adopting BDF schemes and ALE maps for time integration and unfitted finite element method for spatial discretization. The method is applied to various models, including a varying interface problem, a PDE-domain coupled problem, and a problem with topologically changing domain. The ALE-UFE method has the potential for solving a variety of moving-domain problems, such as fluid dynamics and FSI problems.
These will be our future work. 
\begin{appendix} 
\section{Useful estimates of the ALE mapping $\BX^{n,n-1}$}
\begin{lemma}
	Suppose $\Bg^{n,n-1}(\Bx) =\BX_F(t_{n-1};t_n,\Bx)$ and $\BX_F$ is given by \eqref{eq:X}. 
	The Jacobi matrices
	$\bbJ^{n,n-i}=\partial_{\Bx} \BX^{n,n-i}$, $1\leq i\leq k$, of ALE maps admit
	\begin{equation*}
		\|\bbJ^{n,n-i}-\bbI\|_{L^{\infty}(\Omega^n)} \lesssim \tau,\quad 
		\|\Bw_k^n(\Bx,t_n)\|\lesssim \infty.
	\end{equation*}
\end{lemma}
\begin{proof}
	For ease of notation, we write $\Bd_\BX =\BX^{n,n-1}-\BX^{n,n}$.
	Then \eqref{eq:construction of X} implies
	\begin{align*}
		-\Delta \Bd_\BX = \textbf{0} \quad \text{in}\; \Omega^n,\qquad
		\Bd_\BX = \Bg^{n,n-1}-\Bg^{n,n}  \quad \text{on}\; \Gamma^n.
	\end{align*}
	Since $\Bg^{n,n}=\BI$ and both $\Gamma^n$ and $\Bv$ are $C^r$-smooth, we have
	$\|\Bg^{n,n-1}-\Bg^{n,n}\|_{\BH^{r}(\Gamma^n)}\lesssim \tau$. 
	The regularity result of elliptic equations yields
	\begin{equation}\label{eq:estimate of Ex}
		\N{\Bd_\BX}_{\BH^{r+1/2}(\Omega^n)}\leq C \N{\Bg^{n,n-1}-\Bg^{n,n}}_{\BH^{r}(\Gamma^n)}\lesssim \tau.
	\end{equation}
	Using $\BX^{n,n}=\BI$, Sobolev's inequality, and \eqref{eq:estimate of Ex}, we find that
	\begin{equation*}
		\|{\bbJ^{n,n-1}-\bbI}\|_{\BL^\infty(\Omega^n)}=
		\|{\Bd_\BX}\|_{\BW^{1,\infty}(\Omega^n)}\lesssim
		\|{\Bd_\BX}\|_{\BH^{r+1/2}(\Omega^n)}\lesssim \tau.
	\end{equation*}
	By \eqref{eq:multi step} and the chain rule, we obtain
	$\|{\bbJ^{n,n-i}-\bbI}\|_{\BL^\infty(\Omega^n)}\lesssim \tau$ for $2\leq i\leq k$. 
	
	From the definition of $\Bw^{n}_k$ in \eqref{eq:wn},  a simple calculation gives
	\begin{equation*}
		\Bw_k^n(\Bx,t_n) |_{\Gamma^n}= \sum_{i=0}^k (l_n^i)'(t_n)\BX^{n,n-i}|_{\Gamma^n}
		= \frac{1}{\tau}\sum_{i=0}^k \lambda_i^k \BX^{n,n-i}|_{\Gamma^n}=
		\frac{1}{\tau}\sum_{i=0}^k \lambda_i^k \Bg^{n,n-i},
	\end{equation*}
	where $\lambda_i^k$ are the coefficients of BDF-$k$.
	Clearly $\Bw_k^n(\Bx,t_n) |_{\Gamma^n}$ 
	is an approximation to $\partial_t \BX_F(t;t_n,\cdot)|_{t=t_n}$. 
	It is easy to see from \eqref{eq:construction of X} that
	\begin{equation}\label{wn}
		-\Delta \Bw^{n}_k(\Bx,t_n) = 0 \quad \text{in}\;\; \Omega^n,\qquad
		\Bw^{n}_k(\Bx,t_n) =\tau^{-1}\sum_{i=0}^k \lambda_i^k \Bg^{n,n-i}   
		\quad\text{on} \;\; \Gamma^{n}.
	\end{equation}
	By Sobolev's inequalities and  regularity theories of elliptic equations, we obtain
	\begin{align*}
		\|\Bw^{n}_k(t_n)\|_{\BL^\infty(\Omega^n)}
		\lesssim \frac{1}{\tau}\Big\|\sum_{i=0}^k \lambda_i^k \Bg^{n,n-i} \Big\|_{\BH^{r}(\Gamma^n)} \lesssim 
		\|\partial_t \BX_F(t;t_n,\cdot)|_{t=t_n}\|_{\BH^{r}(\Gamma^n)}.  
	\end{align*}
\end{proof}

Since the computations and numerical analysis are performed on the approximate domain $\Omega_{\eta}^n$, which is different from the exact one $\Omega^n$ in general, we have to extend $\BX^{n,n-1}$ from $\Omega^n$ to the fictitious domain $\tilde{\Omega}^n$.

\begin{lemma}
	There exits an extension of $\BX^{n,n-1}\in \BH^{r+1/2}(\Omega^n)$,
	denoted by $\tilde\BX^{n,n-1}\in \BH^{r+1/2}(\tilde\Omega^n)$, such that
	\begin{equation}
		\tilde\BX^{n,n-1}|_{\Omega^n} =\BX^{n,n-1},\qquad
		\|\tilde \BX^{n,n-1}\|_{\BH^{r+1/2}(\tilde{\Omega}^n)} \leq \infty.
	\end{equation}
	Moreover, the Jacobi matrix $\tilde{\bbJ}^{n,n-1} = \partial_{\Bx}\tilde{\BX}^{n,n-1}$ satisfies
	\begin{equation*}
		\tilde{\bbJ}^{n,n-1}|_{\Omega^n} = \bbJ^{n,n-1},\qquad
		\|\tilde\bbJ^{n,n-1}-\bbI\|_{\bbL^{\infty}(\tilde{\Omega}^n)} \lesssim \tau.
	\end{equation*}
\end{lemma}
\begin{proof}
	By the Sobolev extension theorem and \eqref{eq:estimate of Ex}, there exits an extension of $\Bd_{\BX}$, denoted by
	$\widetilde{\Bd_{\BX}}\in \BH^{r+1/2}(\tilde{\Omega}^n)$, such that
	\begin{equation}\label{eq:extension X*}
		\widetilde{\Bd_{\BX}}|_{\Omega^n} = \Bd_\BX,\qquad 
		\|\widetilde{\Bd_\BX}\|_{\BH^{r+1/2}(\tilde{\Omega}^n)}
		\lesssim \|\Bd_\BX\|_{\BH^{r+1/2}(\Omega^n)} \lesssim \tau.
	\end{equation}
	Then the extension of $\BX^{n,n-1}$ is defined by
	$\tilde{{\BX}}^{n,n-1} := \BI +\tau \widetilde{\Bd_\BX}$. It follows that
	\begin{equation}\label{eq:extension X}
		\tilde{\BX}^{n,n-1}|_{\Omega^n} = {\BX}^{n,n-1},\qquad 
		\|\tilde{\BX}^{n,n-1}\|_{\BH^{r+\frac 12}(\tilde{\Omega}^n)}
		\lesssim \text{area}(D) +\tau.
	\end{equation}
	Since $r\geq k+1\geq 3$,
	from Sobolev's inequality and \eqref{eq:extension X*}, for $\mu=0$, $1$, we get
	\begin{equation} \label{eq:linfy of X-x}
		\|\tilde{\BX}^{n,n-1}-\BI\|_{\BW^{\mu,\infty}(\tilde{\Omega}^n)}\lesssim\tau
		\|\widetilde{{\Bd}_{\BX}}\|_{\BH^{2+\mu}(\tilde{\Omega}^n)}\lesssim \tau^2. 
	\end{equation}
	This implies $\big\|\tilde \bbJ^{n,n-1}-\bbI\big\|_{\bbL^\infty(\tilde \Omega^n)}\lesssim \tau$ by noting $\tilde \bbJ^{n,n-1}=\partial_{\Bx} \tilde \BX^{n,n-1}$.
\end{proof}	
\section{Useful estimates of the discrete ALE map $\BX_h^{n,n-i}$}
\begin{lemma}\label{lem:Xh}
	Let $\bbJ_h^{n,n-1}:=\partial_{\Bx} \BX_h^{n,n-1}$ be the Jacobi matrix of $\BX_h^{n,n-1}$. Upon hidden constants independent of $\tau,h,n$, and $\eta$, there hold
	\begin{equation}\label{Xh}
		\|{ \bbJ_h^{n,n-1}-\bbI }\|_{\bbL^\infty(\tilde \Omega^n)}\lesssim\tau,\qquad
		\|{\BX_h^{n,n-1}-\tilde\BX^{n,n-1}}\|_{\BL^\infty(\tilde \Omega^n)}
		\lesssim\tau^{k-1/2}.
	\end{equation}
\end{lemma}
\begin{proof}
	Since  $\Omega_\eta^n \neq \Omega^n$, 
	multiplying both sides of \eqref{eq:construction of X} by $\Bv_h\in\BV(k,\Ct_h^n)$
	and using $\tilde \BX^{n,n-1}\in \BH^{r+1/2}(\tilde \Omega^n)\subset  H^{k+1}(\tilde \Omega^n)$,
	we have
	\begin{equation}\label{eq:extend X}
		\Ca_h^n(\tilde \BX^{n,n-1},\Bv_h) = -(\Delta \tilde{\BX}^{n,n-1},\Bv_h)_{\Omega_\eta^n \backslash \Omega^n}
		+\Cf_{\Gamma_\eta^n}(\tilde \BX^{n,n-1},\Bv_h).
	\end{equation}
	For convenience, let $\Ci_h^n$ be the Lagrange interpolation operator and define
	\begin{equation*}
		\Be = \BX_h^{n,n-1}-\tilde\BX^{n,n-1},\;
		\xibf_h = \BX_h^{n,n-1}-\Ci_h^n \tilde \BX^{n,n-1},\;
		\etabf = \Ci_h^n \tilde \BX^{n,n-1}-\tilde \BX^{n,n-1}.
	\end{equation*}
	It is easy to see that
	$\Be = \xibf_h +\etabf$.
	Subtracting \eqref{eq:extend X} from \eqref{eq:discrte Xh}, one has
	\begin{align}\label{eq:Ahe}
		\Ca_h^n(\Be,\Bv_h)=
		(\Delta \tilde{\BX}^{n,n-1},\Bv_h)_{\Omega_\eta^n\backslash\Omega^n}
		+\Cf_{\Gamma_\eta^n}(\Bg_{\eta}^{n,n-1} -\tilde \BX^{n,n-1},\Bv_h).
	\end{align}
	
	For any $\Bv_h\in \BV(k,\Ct_h^n)$, Poincar$\acute{\text{e}}$ inequality and \eqref{norm-eq0} show 
	\begin{equation}\label{eq:L2 poincare}
		\NLtwov[\tilde \Omega^n]{\Bv_h}\lesssim \SNHonev[\tilde \Omega^n]{\Bv_h}+\mathscr{J}_0^n(\Bv_h,\Bv_h)\lesssim \TN{\Bv_h}_{\Ct_h^n}. 
	\end{equation}
	Since $\text{area}(\Omega_\eta^n\backslash\Omega^n) =\tau^{k+1}$,
	by \cite[Lemma~A4]{ma21} and \eqref{eq:L2 poincare}, we have
	\begin{align}
		(\Delta \tilde{\BX}^{n,n-1},\Bv_h)_{\Omega_\eta^n \backslash \Omega^n}
		&\lesssim \tau^{k+1} \|{\tilde \BX^{n,n-1}}\|_{\BH^{3}(\tilde{\Omega}^n)}
		\||{\Bv_h}|\|_{\Ct_h^n}.
		\label{eq:errX1}
	\end{align}
	Since the approximate boundary satisfies $\text{dist}(\Gamma^n,\Gamma_{\eta}^n)\lesssim \tau^{k+1}$, and 
	$\Bg_{\eta}^{n,n-1}$ is the $(k+1)^{\text{th}}$-order approximation to $\Bg^{n,n-1}$, there holds
	\begin{equation}\label{eq:eq:trace Xtau-X}
		\|\Bg_{\eta}^{n,n-1}-\tilde\BX^{n,n-1}\|_{L^2(\Gamma_\eta^n)}
		\lesssim\;\tau^{k+1}.
	\end{equation}
	From the trace inequality, inverse estimate, and \eqref{eq:eq:trace Xtau-X}, we have
	\begin{align}
		\Cf_{\Gamma_\eta^n}(\BX_\tau^{n,n-1} -\tilde \BX^{n,n-1},\Bv_h)
		\lesssim \tau^{k+\frac 12}\||{\Bv_h}|\|_{\Ct_h^n}.
		\label{eq:errX2}
	\end{align}
	Insert \eqref{eq:errX1} and \eqref{eq:errX2} into \eqref{eq:Ahe}. 
	By  Lemma~\ref{lem:Ah} and interpolation error estimates, we obtain
	\begin{align}
		\||{\xibf_h}|\|^2_{\Ct_h^n} \lesssim \mathscr{A}_h^n(\xibf_h,\xibf_h)=\Ca_h^n(\Be,\xibf_h)
		-\Ca_h^n(\etabf,\xibf_h)
		\lesssim (\tau^{k+\frac 12}+h^k)\||{\xibf_h}|\|_{\Ct_h^n}.\label{eq:err-xih}
	\end{align}
	Thus together with \eqref{norm-eq0} and \eqref{eq:err-xih}, one has
	\begin{equation}\label{eq:err-xih2}
		|{\xibf_h}|_{\BH^1(\tilde{\Omega}^n)}\lesssim \||{\xibf_h}|\|_{\Ct_h^n}\lesssim h^k+\tau^{k+\frac 12}.
	\end{equation}
	These yield $|{\Be}|_{\BH^1(\tilde{\Omega}^n)} \lesssim h^k + \tau^{k+\frac 12}$.
	Since $\|{\tilde \BX^{n,n-1}}\|_{\BW^{k,\infty}(\tilde \Omega^n)}\lesssim
	\|{\tilde \BX^{n,n-1}}\|_{\BH^{r+\frac 12}(\tilde \Omega^n)}$,
	inverse estimates imply
	\begin{align}
		\|{\Be}\|_{\BW^{1,\infty}(\tilde \Omega^n)}
		\lesssim h^{-1}\|{\xibf_h}\|_{\BH^{1}(\tilde{\Omega}^n)}+ h^{k-1} \|{\tilde \BX^{n,n-1}}\|_{\BH^{r+\frac 12}(\tilde \Omega^n)}
		\lesssim h^{k-1} +\tau^{k-\frac 12}.\label{eq:err e W1infty}	
	\end{align} 
	Let $\bbJ_h^{n,n-1}=\partial_{\Bx} \tilde \BX_h^{n,n-1}$.
	From  \eqref{eq:err e W1infty} and \eqref{eq:linfy of X-x}, we have
	\begin{align}
		\|{\bbJ_h^{n,n-1}-\bbI}\|_{\BL^\infty(\tilde{\Omega}^n)}
		&\leq \|{\bbJ_h^{n,n-1}-\tilde\bbJ^{n,n-1}}\|_{\BL^\infty(\tilde{\Omega}^n)}
		+\|{\tilde \bbJ^{n,n-1}-\bbI}\|_{\BL^\infty(\tilde{\Omega}^n)}
		\leq \|{\Be}\|_{\BW^{1,\infty}(\tilde{\Omega}^n)} 
		+\tau\|{\BX^{n,n-1}}\|_{\BH^{r+\frac 12}(\Omega^n)}\lesssim \tau.\notag
	\end{align}
	
	From Sobolev's inequality and inverse estimates, for any $v_h\in H^{1}(\tilde \Omega^n)$, we get
	\begin{equation}\label{eq:infty to H1}
		\NLinf[\tilde \Omega^n]{v_h}\lesssim \N{v_h}_{W^{1,4}(\tilde \Omega^n)}
		\lesssim h^{-\frac 12}\NHone[\tilde \Omega^n]{v_h}.
	\end{equation}
	Moreover, it yields from \eqref{eq:L2 poincare}, \eqref{eq:err-xih2}, \eqref{eq:infty to H1} that
	\begin{align*}
		\N{\Be}_{\BL^\infty(\tilde \Omega^n)} &\leq 
		\N{\xibf_h}_{\BL^{\infty}(\tilde \Omega^n)}
		+\N{\etabf}_{\BL^{\infty}(\tilde \Omega^n)}
		\lesssim h^{-\frac 12}\NHone[\tilde \Omega^n]{\xibf_h}
		+h^{k} \N{\BX^{n,n-1}}_{\BW^{k,\infty}(\tilde \Omega^n)}\\
		&\lesssim h^{-\frac 12}\TN{\xibf_h}_{\Ct_h^n}
		+h^{k}\lesssim h^{k-\frac 12}+h^{-\frac 12}\tau^{k+\frac 12}.
	\end{align*}
	The proof is finished by assuming $h=O(\tau)$.
\end{proof}    

\begin{corollary}\label{corol:bound of Jh}
	Let $\Bw^{n}_{k,h}$ be defined in \eqref{eq: ALE Ah} and write $\bbJ_h^{n,n-i} := \partial_{\Bx} \BX^{n,n-i}_h$ and 
	$\bbJ_h^{n-i,n} := (\bbJ_h^{n,n-i})^{-1}$ for $1\leq i\leq k$. Then $\|\Bw_{k,h}^n\|_{\BL^{\infty}(\Omega_{\eta}^n)}\lesssim \N{\Bw^{n}_{k}}_{\BH^{r+1/2}(\Omega^n)}$ and
	\begin{align}
		&\|\bbJ_h^{n,n-i}-\bbI\|_{\bbL^{\infty}(\tilde\Omega^n)} 
		\lesssim \tau , \hspace{19mm}
		\Det (\bbJ_h^{n,n-i}) = 1+O(\tau), \label{eq:Jacobin Xh}\\
		&\|{\bbJ_h^{n-i,n}-\bbI}\|_{\bbL^\infty(\BX_h^{n,n-i}(\Omega_\eta^n))}
		\lesssim \tau, \hspace{12mm}
		\Det (\bbJ_h^{n-i,n}) = 1+O(\tau). \label{eq:inverse Jacobin Xh}
	\end{align}
\end{corollary}
\begin{proof}
	Thanks to \eqref{eq:delta omega} and Lemma~\ref{lem:Xh}, we have
	$\BX_h^{n,n-j}(\Omega_\eta^{n}) \subset \tilde{\Omega}^{n-j}$ for $1\leq j<k$.
	So \eqref{def:Xh n-i} is well defined.
	The proofs of \eqref{eq:Jacobin Xh}--\eqref{eq:inverse Jacobin Xh} are direct consequences of \eqref{Xh}. The stability of Sovolev extension shows that
	$\tilde \Bw^{n}_k = \frac{1}{\tau}\sum\limits_{i=0}^k \lambda_i^k \tilde\BX^{n,n-i}$ satisfies
	\begin{equation*}
		\tilde\Bw^{n}_k|_{\Omega^n} = \Bw^{n}_k,\qquad
		\|\tilde \Bw^{n}_k\|_{\BH^{r+1/2}(\tilde\Omega^n)} 
		\lesssim \| \Bw^{n}_k\|_{\BH^{r+1/2}(\Omega^n)}.
	\end{equation*}
	Using Lemma~\ref{lem:Xh}, we have
	\begin{equation}\label{eq:bound of Bw}
		\|\Bw_{k,h}^n-\tilde\Bw^{n}_k\|_{\BL^\infty(\Omega_{\eta}^n)} \lesssim
		\tau^{-1}\sum_{i=0}^k
		\|\BX_h^{n,n-i}-\tilde\BX^{n,n-i}\|_{\BL^{\infty}(\Omega_{\eta}^n)}
		\lesssim \tau^{k-3/2}.
	\end{equation}
	The proof is finished by the triangle inequality.
\end{proof}

Finally, we present estimates of the discrete ALE mapping
by using Lemma~\ref{lem:Xh} and Corollary~\ref{corol:bound of Jh}.
\begin{lemma}\label{lem:Uh}
	Assuming $\BX_h^{n,n-1}$ satisfies \eqref{Xh}
	and $\tau = O(h)$,
	there exits a constant $C$ independent of $n,\tau$ such that, for any $v_h\in V(k,\Ct_h^{n-i})$, and $\mu=0,1$,
	\begin{align}\label{eq:UX}
		\|{\nabla^{\mu}
			\big(v_h \circ \BX_h^{n,n-i}\big)}\|^2_{L^2(\Omega_\eta^n)}&\leq\;  
		\|{\nabla^{\mu}(v_h \circ \BX_h^{n-l,n-i})}\|^2_{L^2(\Omega_\eta^{n-l})}
		+Ch \|\nabla^{\mu} v_h\|_{L^2(\tilde{\Omega}^{n-i})}^2,\\
		\|{v_h\circ \BX_h^{n,n-1}}\|^2_{L^2(\Gamma_\eta^n)}&\leq 
		(1+C\tau)\NLtwo[\Gamma_\eta^{n-1}]{v_h}^2 +C \tau^{k-1}\NHone[\tilde{\Omega}^{n-1}]{v_h}^2, \label{eq:L2 trace}\\
		\big|{v_h\circ \BX_h^{n,n-1}}\big|^2_{H^1(\Gamma_\eta^n)} &\lesssim h^{-1}
		\SNHone[\tilde \Omega^{n-1}]{v_h}^2. \label{eq:H1 trace}
	\end{align}
\end{lemma}

\section{The stability of numerical solutions}
\label{sec:The proof of the stability for the discrete solutions}

Since the computational domain is time-varying, we have to deal with the issue 
that $u^{n-j}_h\notin V(k,\Ct^n_h)$ for $1\le j\le k$ 
when proving the stability and convergence of the discrete solution. 
To overcome this difficulty,
we introduce a modified Ritz projection operator $\Cp^n_h:Y(\Omega^n_{\eta})\to V(k,\Ct^n_h)$  proposed in \cite{ma21} which 
satisfies
\begin{align}	\label{eq:Pnh}
	\mathscr{A}^n_h(\Cp^n_hw,v_h)  =  a^n_h(w,v_h) \qquad
	\forall\, v_h\in V(k,\Ct^n_h),
\end{align}
where $Y(\Omega_\eta^{n}) := \big\{v\in \Hone[\Omega^n_{\eta}]:\TN{v}_{\Omega^n_{\eta}}<\infty\big\}$.
It satisfies
\begin{align}\label{Rw-0} 
	&\TN{\Cp_h^n w}_{\Omega_\eta^n}
	\lesssim \TN{w}_{\Omega_\eta^n},\qquad 
	\NLtwo[\Omega^n_{\eta}]{w-\Cp^n_h w}
	\lesssim  h \TN{w}_{\Omega^n_{\eta}},\quad
	\forall w\in Y(\Omega^n_{\eta}).
\end{align}
\subsection{Proof of Theorem~\ref{thm:uh-stab}}
\begin{proof}
	We only prove the theorem for $k=4$.
	The proofs for other cases are similar.
	The rest of the proof is parallel to the proof of \cite[Theorem~5.2]{ma21},
	we just sketch the proof and omit details. 
	
	\textbf{Step1: Selection of a particular test function.}
	Write $\tilde{U}^{n-1,n}_h=\Cp^n_h(U^{n-1,n}_h)$.
	Since $U^{n-1,n}_h\notin V(k,\Ct^n_h)$, we choose $v_h=2u^n_h -\tilde U^{n-1,n}_h$ as a test function, the discrete problem for $k=4$ has the form
	\begin{equation}  
		\sum_{i=1}^5 
		\|\Psibf_i^4 (\underline{\BU_h^{n}})\|_{L^2(\Omega^n_{\eta})}^2 
		- \sum_{i=1}^4 
		\|\Phibf_i^4 (\underline{\BU_h^{n}})\|^2_{L^2(\Omega^n_{\eta})} 
		+ \tau \mathscr{A}^n_h(u_h^n,v_h)+\tau \mathscr{B}_h^n(u_h^n,v_h)
		= \tau R^n_1 +R^n_2,   \label{eq:stab0}
	\end{equation}
	where $\mathscr{B}_h^n(u_h^n,v_h) = (-\nabla \Bw_{k,h}^n\cdot \nabla u_h^n,v_h)_{\Omega_{\eta}^n}$, 
	we have used \cite[Section~2 and Appendix~A]{liu13}, and 
	\begin{align*}
		&\Psibf^4_i \big(\ul{\BU^{n}_h}\big) =\sum_{j=1}^{i} c_{i,j}^4 U_h^{n+1-j,n},\quad
		\Phibf_i^4 \big(\ul{\BU^{n}_h}\big)= \sum_{j=1}^{i} c_{i,j}^4 U_h^{n-j,n},\\
		&R^n_1=\big(f^n,2u^n_h -\tilde U^{n-1,n}_h\big)_{\Omega^n_{\eta}},\quad
		R^n_2 = \big(\Lambda^4\ul{\BU_h^{n}},\tilde U^{n-1,n}_h - U^{n-1,n}_h\big)_{\Omega^n_{\eta}}.
	\end{align*}
	The trace inequality and inverse estimates show $
	\|\partial_\Bn u^n_h\|_{L^2(\Gamma^n_{\eta})} 
	\lesssim h^{-\frac 12}|u^n_h|_{H^1(\tilde{\Omega}^n)}$. For any $\varepsilon\in (0,1)$, 
	by the definitions of $\Cp^n_h$ and the Cauchy-Schwarz inequality, one has
	\begin{align}
		&\mathscr{A}^n_h(u_h^n,\,2u_h^n-\tilde U^{n-1,n}_h) 
		=\,2\mathscr{A}^n_h(u_h^n,u_h^n)
		-a^n_h(u_h^n, U^{n-1,n}_h)  \notag \\
		&\,\ge \frac{3}{2}\TN{u_h^n}^2_{\Ct_h^n} 
		-\frac{5}{2\varepsilon h}
		\|{u^{n}_h}\|^2_{L^2(\Gamma^{n}_\eta)}
		-C\varepsilon|{u^n_h}|^2_{H^1(\tilde{\Omega}^n)} 
		-\frac{1}{2}|{U^{n-1,n}_h}|^2_{H^1(\Omega^{n}_\eta)}
		-\frac{\varepsilon h}{2} |{U_h^{n-1,n}}|^2_{H^1(\Gamma_\eta^n)}
		- \frac{\gamma_0+\varepsilon^{-1}}{2h}
		\|{U^{n-1,n}_h}\|^2_{L^2(\Gamma^{n}_\eta)}.   \label{stab-est1}
	\end{align}		
	Since $\Bw_{k,h}^n$ is  bounded with $C_w=\|\Bw_{k,h}^n\|_{\BL^{\infty}(\Omega_{\eta}^n)}$ and
	$U^{n-1,n}_h\in Y(\Omega^n_\eta)$,  
	we infer from
	\eqref{Rw-0} and the Cauchy-Schwarz inequality that
	\begin{align}
		\mathscr{B}_h^n(u_h^n,v_h)\leq\,& 
		\frac{|{u_h^n}|^2_{H^1(\Omega_\eta^n)}}{4}+
		8 C_{w}^2\big(\|{u_h^n}\|^2_{L^2(\Omega_\eta^n)}
		+\|{U_h^{n-1,n}}\|^2_{L^2(\Omega_\eta^n)}
		+Ch^2\||{U_h^{n-1,n}}\||_{\Omega_\eta^n}^2\big),\notag  \\
		R^n_1\le\, & 2\|{f^n}\|^2_{L^2(\Omega^n_\eta)} 
		+\|{u^n_h}\|^2_{L^2(\Omega^n_\eta)}
		+\|{U^{n-1,n}_h}\|^2_{L^2(\Omega^n_\eta)}
		+C h^2\||{U^{n-1,n}_h}\||_{\Omega^n_\eta}^2,\notag \\
		R^n_2 \le\, & C \tau \varepsilon^{-1}\sum_{j=0}^4
		\|{U^{n-j,n}_h}\|^2_{L^2(\Omega^n_{\eta})}
		+\frac{\tau \varepsilon}{4}\||{U^{n-1,n}_h}\||_{\Omega^n_\eta}^2. \notag 
	\end{align}

	\textbf{Step2: Estimate of $\Phi_i^4(\ul{\BU_h^n})$.}
	We deduce from the inverse estimates and Lemma~\ref{lem:Uh} that
	\begin{align} 
		\|{\nabla^\mu U^{n-j,n}_h}\|^2_{\BL^2(\Omega^{n}_\eta)}
		\le\,&  
		\|{\nabla^\mu u^{n-j}_h}\|^2_{\BL^2(\Omega^{n-j}_\eta)}
		+C\tau 
		\|{\nabla^\mu u^{n-j}_h}\|^2_{\BL^2(\tilde\Omega^{n-j})}, \quad \mu=0,1, \\
		\|{U^{n-1,n}_h}\|^2_{L^2(\Gamma^{n}_\eta)}
		\le\,& (1+C\tau)\|{u^{n-1}_h}\|^2_{L^2(\Gamma^{n-1}_\eta)}
		+C\tau^3\|{u^{n-1}_h}\|^2_{H^1(\tilde\Omega^{n-1})},\\
		|{U_h^{n-1,n}}|_{H^1(\Gamma_\eta^n)}\lesssim\,& h^{-1/2}
		\|{u_h^{n-1}}\|_{H^1(\tilde \Omega^{n-1})}. \label{stab-est7}
	\end{align}
	Note that $\Phibf_i^4\big(\underline{\BU_h^n}\big)
	=\Psibf_i^4\big(\underline{\BU_h^{n-1}}\big)\circ \BX_{h}^{n,n-1}$. It is easy to see from Lemma~\ref{lem:Uh} that
	\begin{align}
		\|{\Phibf_i^4(\underline{\BU_h^n})}\|^2_{L^2(\Omega_{\eta}^n)}
		\le\, \|{\Psibf_i^4(\underline{\BU_h^{n-1}})}\|^2_{L^2(\Omega_{\eta}^{n-1})}  
		+ C\tau \sum_{j=1}^i 
		\|{u_h^{n-j}}\|^2_{L^2(\tilde{\Omega}^{n-j})} .
		\label{stab-est0}
	\end{align}
	
	\textbf{Step3: Application of Gronwall's inequality.}
	Substituting \eqref{stab-est1}--\eqref{stab-est0} into \eqref{eq:stab0}, using $O(h)=\tau\le h \ll\varepsilon\ll 1$, $C_w^2 h\leq \varepsilon$,  applying~\eqref{norm-eq0}, 	
	and taking the sum of the inequalities over $4\le n\le m$, we end up with 
	\begin{align*} 
		&\sum_{i=1}^4\|
		{\Psibf_i^4(\underline{\BU_h^m})}\|^2_{L^2(\Omega_\eta^m)}
		+\frac{3}{2}\tau\sum_{n=4}^m\||{u_h^n}\||^2_{\Ct_h^n} \\
		\le\,& \sum_{n=4}^m\tau
		(1+C\varepsilon) \||u_h^n\||_{\Ct_h^n}^2 
		+  \sum_{n=4}^m\tau
		\Big[ C\|{u^n_h}\|^2_{L^2(\Omega^n_\eta)} +2\|{f^n}\|^2_{L^2(\Omega^n_\eta)}
		+\big(\frac{6+C\tau}{2\varepsilon h}-\frac{\gamma_0}{2h}\big)
		\|u_h^{n}\|_{L^2(\Gamma_{\eta}^n)}^2 \Big]  
		+C_0,
	\end{align*}
	after careful calculation,
	where $C_0$ is related to the initial solutions, $C_0= C\sum_{i=0}^3\tau\||{u^{i}_h}\||^2_{\Ct_h^i} +\sum_{i=1}^4\|{\Psibf_i^{4}(\underline{\BU_h^3})}\|^2_{L^2(\Omega^3_\eta)}$.
	From \cite[Table~2.2]{liu13}, we know $\Psibf_1^{4}\big(\underline{\BU_h^{m}}\big)=c^4_{1,1}u^m_h$ ($c_{1,1}^4>0$). 
	Finally, we choose $\varepsilon$ small enough such that 
	$C\varepsilon <1/2$ 
	and $\gamma_0$ large enough 
	such that $\gamma_0 \geq (6+C\tau)/\varepsilon$. 
	Then the proof is finished 
	by using Gronwall's inequality.
\end{proof}

\end{appendix}


\begin{thebibliography}{10}
\bibitem{adj19}
{ S.~Adjerid and K.~Moon}, {   An immersed discontinuous Galerkin method for acoustic wave propagation in inhomogeneous media}, SIAM J. Sci. Comput.,
41 (2019), pp.~A139--A162.

\bibitem{ale99}
{  A.~M. Andrew}, {   Level set methods and fast marching methods}, Cambridge University Press, Cambridge, UK, 2~ed., 1999.
\bibitem{asc95}
{  U.~M.~Ascher and S. J. Ruuth, and B. T. Wetton}, {   Implicit-explicit methods for time-dependent partial differential equations}, SIAM J. Numer. Anal., 32(1995), pp.~797-823.



\bibitem{bur22}
{  E.~Burman, S.~Frei, and A.~Massing}, {   Eulerian time-stepping schemes
	for the non-stationary stokes equations on time-dependent domains}, Numer.
Math., 150 (2022), pp.~423--478.

\bibitem{don82}
{  J.~Donea, S.~Giuliani, and J.~P. Halleux}, {   An arbitrary
	lagrangian-eulerian finite element method for transient dynamic
	fluid-structure interactions}, Comput. Methods Appl. Mech. Engrg., 33 (1982),
pp.~689--723.

\bibitem{far01}
{  C.~Farhat, P.~Geuzaine, and C.~Grandmont}, {   The discrete geometric
	conservation law and the nonlinear stability of ALE schemes for the solution
	of flow problems on moving grids}, J. Comput. Phys., 174 (2001), pp.~669--694.

\bibitem{for99}
{  L.~Formaggia and F.~Nobile}, {   A stability analysis for the arbitrary
	Lagrangian Eulerian formulation with finite elements}, East-West J. Math., 7
(1999), pp.~105--132.

\bibitem{fri09}
{  T.~P. Fries and A.~Zilian}, {   On time integration in the XFEM}, Internat. J.
Numer. Methods Engrg., 79 (2009), pp.~69--93.

\bibitem{geu03}
{  P.~Geuzaine, C.~Grandmont, and C.~Farhat}, {   Design and analysis of ALE
	schemes with provable second-order time-accuracy for inviscid and viscous
	flow simulations}, J. Comput. Phys., 191 (2003), pp.~206--227.

\bibitem{guo21}
{  R.~Guo}, {   Solving parabolic moving interface problems with dynamical
	immersed spaces on unfitted meshes: fully discrete analysis}, SIAM J. Numer. Anal., 59 (2021), pp.~797--828.

\bibitem{guz18}
{  J.~Guzmán and M.~Olshanskii}, {   Inf-sup stability of geometrically
	unfitted Stokes finite elements}, Math. Comp., 87 (2018), pp.~2091--2112.

\bibitem{hansbo15}
{  P.~Hansbo, M.~G. Larson, and S.~Zahedi}, {   Characteristic cut finite
	element methods for convection–diffusion problems on time dependent
	surfaces}, Comput. Methods Appl. Mech. Engrg., 293 (2015), pp.~431--461.

\bibitem{hir74}
{  C.~W. Hirt, A.~A. Amsden, and J.~L. Cook}, {   An arbitrary
	lagrangian-eulerian computing method for all flow speeds}, J. Comput. Phys., 14 (1974), pp.~227--253.

\bibitem{hug16}
{  T.~J. Hughes, W.~K. Liu, and T.~K. Zimmermann}, {   Lagrangian-Eulerian
	finite element formulation for incompressible viscous flows}, Comput. Methods
Appl. Mech. Engrg., 29 (1981), pp.~329--349.

\bibitem{leh19}
{  C.~Lehrenfeld and M.~Olshanskii}, {   An Eulerian finite element method
	for PDEs in time-dependent domains}, ESAIM Math. Model. Numer. Anal., 53 (2019),
pp.~585--614.

\bibitem{leh13}
{  C.~Lehrenfeld and A.~Reusken}, {   Analysis of a nitsche XFEM-DG
	discretization for a class of two-phase mass transport problems}, SIAM
J.Numer. Anal., 51 (2013), pp.~958--983.

\bibitem{li03}
{  C.~Y. Li, S.~V. Garimella, and J.~E. Simpson}, {   Fixed-grid
	front-tracking algorithm for solidification problems, part I: Method and
	validation}, Numerical Heat Transfer, Part B: Fundamentals, 43 (2003),
pp.~117--141.

\bibitem{liu13}
{  J.~Liu}, {   Simple and efficient ALE methods with provable temporal
	accuracy up to fifth order for the Stokes equations on time varying domains},
SIAM J. Numer. Anal., 51 (2013), pp.~743--772.

\bibitem{lou21}
{  Y.~Lou and C.~Lehrenfeld}, {   Isoparametric unfitted BDF--finite element
	method for PDEs on evolving domains}, SIAM J. Numer. Anal., 60 (2022),
pp.~2069--2098.

\bibitem{ma23}
{  C.~Ma, T.~Tian, and W.~Zheng}, {   High-order unfitted characteristic
	finite element methods for moving interface problem of Oseen equations}, J.
Comput. Appl. Math., 425 (2023), ~115028.

\bibitem{ma21}
{  C.~Ma, Q.~Zhang, and W.~Zheng}, {   A fourth-order unfitted
	characteristic finite element method for solving the advection-diffusion
	equation on time-varying domains}, SIAM J. Numer. Anal., 60 (2022),
pp.~2203--2224.

\bibitem{maz21}
{  C.~Ma and W.~Zheng}, {   A fourth-order unfitted characteristic finite
	element method for free-boundary problems}, J. Comput. Phys., 469 (2022),
~111552.

\bibitem{mas97}
{  A.~Masud and T.~J. Hughes}, {   A space-time Galerkin/least-squares
	finite element formulation of the Navier-Stokes equations for moving domain
	problems}, Comput. Methods Appl. Mech. Engrg., 146 (1997), pp.~91--126.

\bibitem{rich17}
{  T.~Richter}, {   Fluid-structure interactions: models, analysis and
	finite elements}, Springer, 2017.

\bibitem{tez92}
{  T.~E. Tezduyar, M.~Behr, S.~Mittal, and J.~Liou}, {   A new strategy for
	finite element computations involving moving boundaries and interfaces—the
	deforming-spatial-domain/space-time procedure: II. computation of
	free-surface flows, two-liquid flows, and flows with drifting cylinders},
Comput. Methods Appl. Mech. Engrg., 94 (1992), pp.~353--371.

\bibitem{von23}
{  H.~von~Wahl and T.~Richter}, {   Error analysis for a parabolic PDE model
	problem on a coupled moving domain in a fully Eulerian framework}, SIAM J.
Numer. Anal., 61 (2023), pp.~286--314.

\bibitem{wah20}
{  H.~von Wahl, T.~Richter, and C.~Lehrenfeld}, {   An unfitted Eulerian
	finite element method for the time-dependent Stokes problem on moving
	domains}, IMA J. Numer. Anal., 42 (2022), pp.~2505--2544.

\bibitem{zha18}
{  Q.~Zhang}, {   Fourth-and higher-order interface tracking via mapping and
	adjusting regular semianalytic sets represented by cubic splines}, SIAM J.
Sci. Comput., 40 (2018), pp.~A3755--A3788.

\bibitem{zun13}
{  P.~Zunino}, {   Analysis of backward Euler/Extended finite element
	discretization of parabolic problems with moving interfaces}, Comput. Methods
Appl. Mech. Engrg., 258 (2013), pp.~152--165.

\end{thebibliography}
\end{document}